\documentclass{amsart}
\usepackage{amsmath,amssymb,amsthm,amsfonts,enumerate,array,color,lscape,fancyhdr,layout,pst-all}

\usepackage[english]{babel}
\usepackage[latin1]{inputenc}
\usepackage{young}

\usepackage[hcentermath]{youngtab}

\usepackage{graphicx}
\usepackage{float}
\usepackage{pstricks}

\newcounter{numberofremark}
\newcommand\nothing[1]{}
\usepackage[active]{srcltx}
\usepackage[plainpages]{hyperref}

\newcommand{\dcl}{\DeclareMathOperator}
\dcl\cdet{cdet} \dcl\Sp{Specm} \dcl\depth{depth} \dcl\im{Im} \dcl\Span{span} \dcl\Ker{Ker} \dcl\Specm{Specm}
\dcl\Supp{Supp} \dcl\codim{codim} \dcl\Y{Y} \dcl\gl{\mathfrak{gl}}    \dcl\U{U} \dcl\T{T}
\dcl\qdet{qdet} \dcl\sgn{sgn} \dcl\gr{gr} \dcl\diag{diag}
\dcl\g{\mathfrak{g}} \dcl\C{\mathbb C} \dcl\dd{{\mathrm d}}

\newcommand\sn{{\mathsf n}}
\newcommand\sm{{\mathsf m}}

\newcommand\Ga{{\Gamma}}

\def\cM{\mathcal M}

\newlength\yStones
\newlength\xStones
\newlength\xxStones

\makeatletter
\def\Stones{\pst@object{Stones}}
\def\Stones@i#1{%
  \pst@killglue%
  \begingroup%
  \use@par%
  \setlength\xxStones{\xStones}%
  \expandafter\Stones@ii#1,,\@nil
  \endgroup
  \global\addtolength\xStones{0.6cm}%
  \global\addtolength\yStones{-7.5mm}}%
\def\Stones@ii#1,#2,#3\@nil{%
  \rput(\xxStones,\yStones){%
    \psframebox[framesep=0]{%
      \parbox[c][6mm][c]{11mm}{\makebox[11mm]{$#1$}}}}%
  \addtolength\xxStones{1.2cm}%
  \ifx\relax#2\relax\else\Stones@ii#2,#3\@nil\fi}
\makeatother

\def\Stone#1{\fbox{\makebox[11mm]{\strut#1}}\kern2pt}

\newtheorem{theorem}{Theorem}[section]
\newtheorem{lemma}[theorem]{Lemma}
\newtheorem{corollary}[theorem]{Corollary}
\newtheorem{proposition}[theorem]{Proposition}

\newtheorem{remark}[theorem]{Remark}

\newtheorem{definition}[theorem]{Definition}

\begin{document}

\title{Singular Gelfand-Tsetlin modules of $\displaystyle \gl(n)$}

\author{Vyacheslav Futorny}

\address{Instituto de Matem\'atica e Estat\'istica, Universidade de S\~ao
Paulo,  S\~ao Paulo SP, Brasil} \email{futorny@ime.usp.br,}
\author{Dimitar Grantcharov}
\address{\noindent
University of Texas at Arlington,  Arlington, TX 76019, USA} \email{grandim@uta.edu}
\author{Luis Enrique Ramirez}
\address{Instituto de Matem\'atica e Estat\'istica, Universidade de S\~ao
Paulo,  S\~ao Paulo SP, Brasil} \email{luiser@ime.usp.br,}

\begin{abstract}

The classical Gelfand-Tsetlin formulas provide a basis  in terms of tableaux and an explicit action of the generators of $\mathfrak{gl} (n)$  for every irreducible finite-dimensional  $\mathfrak{gl} (n)$-module. These formulas can be used to define a $\mathfrak{gl} (n)$-module structure on some infinite-dimensional modules - the so-called generic Gelfand-Tsetlin modules. The generic Gelfand-Tsetlin  modules are convenient to work with since  for every generic tableau there exists a unique irreducible generic Gelfand-Tsetlin module containing this tableau as a basis element. In this paper we initiate the systematic study of a large class of non-generic Gelfand-Tsetlin modules - the class of  $1$-singular Gelfand-Tsetlin modules. An explicit tableaux realization and the action of $\mathfrak{gl} (n)$ on these modules is  provided  using a new construction which we call derivative tableaux. Our construction of $1$-singular modules provides a large family of new irreducible Gelfand-Tsetlin modules of $\gl (n)$, and is a part of the classification of all such irreducible modules for $n=3$.
\end{abstract}

\maketitle
\section{Introduction}

A recent major breakthrough in the representation theory was the  classification of the irreducible weight modules with finite weight multiplicities of all finite-dimensional reductive complex Lie algebras. The classification result was obtained in two steps: first in  \cite{Fe}, using parabolic induction functors, the problem was reduced to simple Lie algebras of type $A$ and $C$; and then in \cite{M}, the classification was completed and an explicit realization of the irreducibles was obtained. Recall that $M$ is weight module of a Lie algebra ${\mathfrak g}$ with fixed Cartan subalgebra  $\mathfrak h$ if $M$ is $\mathfrak h$-diagonalizable, and that the dimension of the weight space $M_\lambda=\{v\in M\mid hv=\lambda(h)v\
(\forall h\in {\mathfrak h})\}$ is called the weight multiplicity of $\lambda \in {\mathfrak h}^*$.

On the other hand, the problem of classifying all irreducible weight modules (possibly with infinite weight multiplicities) is still largely open.  A natural large class of such modules  consists of the so-called Gelfand-Tsetlin modules. The Gelfand-Tsetlin modules are defined by  generalizing a classical construction of Gelfand and Tsetlin that provides a convenient basis for every simple finite dimensional representation of a simple classical Lie algebra. The theory of general Gelfand-Tsetlin modules, especially for Lie algebras of type $A$,  has attracted considerable attention in the last 30 years and 
have been studied in \cite{DFO2}, \cite{DFO1}, \cite{GT}, \cite{Maz1}, \cite{Maz2}, \cite{m:gtsb}, \cite{Zh}, among others.  

In this paper we consider Gelfand-Tsetlin modules for Lie algebras of type $A$ and  for simplicity work with $\gl(n)$ instead of $\mathfrak{sl}(n)$.  The Gelfand-Tsetlin modules of $\gl(n)$  by definition are modules that admit a basis of common eigenvectors of a fixed maximal commutative subalgebra $\Gamma$ of the universal enveloping algebra $U(\gl (n))$ of $\gl(n)$. The algebra $\Gamma$  is  called the Gelfand-Tsetlin subalgebra of  $U(\gl (n))$ and has numerous applications that extend beyond the theory of Gelfand-Tsetlin modules.  Gelfand-Tsetlin subalgebras were considered in \cite{Vi} in  connection with subalgebras of maximal Gelfand-Kirillov dimension in the universal enveloping
algebra of a simple Lie algebra. Furthermore, these subalgebras are related to: general hypergeometric functions on the complex Lie group $GL(n)$, \cite{Gr1},\cite{Gr2}; solutions of the
Euler equation, \cite{Vi}; and problems in classical mechanics in general,  \cite{KW-1}, \cite{KW-2}.

It is well known that the Gelfand-Tsetlin subalgebra $\Gamma$ has a simple spectrum on any  irreducible finite-dimensional module, that is, the
characters of $\Gamma$ separate the basis elements of such module.  However, this property does not longer hold  for infinite-dimensional Gelfand-Tsetlin modules, in which case  Gelfand-Tsetlin characters may occur with multiplicities.  These Gelfand-Tsetlin multiplicities are always finite by \cite{Ovs}, and bounded by
\cite{FO2}. The so-called generic Gelfand-Tsetlin modules (for the explicit definition see Section 3) have the convenient property that all their  Gelfand-Tsetlin multiplicities are $1$. Gelfand-Tsetlin modules with multiplicities $1$  have been studied in several papers. In \cite{GG}, using the classical Gelfand-Tsetlin formulas, a Gelfand-Tsetlin basis was constructed for some of these modules. Later,  in  \cite{LP1, LP2}   the construction of  \cite{GG} was extended to all such modules for $n=3$. The existence of higher Gelfand-Tsetlin multiplicities is a major obstacle to study and explicitly construct tableaux-type basis for  such modules in general.
So far, the only known examples of  Gelfand-Tsetlin modules with higher multiplicities and with basis of tableaux, are examples of Verma modules considered in \cite{FJMM}.

In this paper we initiate the  systematic study  of non-generic Gelfand-Tsetlin modules, i.e. of singular Gelfand-Tsetlin modules. Singular modules are those for which the denominators in the Gelfand-Tsetlin formulas may vanish, and  we focus on the case of $1$-singular modules, that is, the case when only simple poles are allowed. In order to deal with the singularities of the coefficients we introduce a new type of tableaux - the derivative Gelfand-Tsetlin tableaux, or simply, the derivative tableaux.  We expect that the derivative tableaux construction can be extended to more general singular Gelfand-Tsetlin modules using differential operators of higher order.

The main results of present paper can be separated into two 
components:

(1)  With the aid of the derivative tableaux  we construct  a universal Gelfand-Tsetlin module for any $1$-singular Gelfand-Tsetlin character. This is achieved by providing  an explicit tableaux-type basis and defining the action of the generators of $\gl(n)$ in the spirit of the original work of Gelfand and Tsetlin. Our construction extends and generalize the previous  works on finite-dimensional and on generic Gelfand-Tsetlin modules.  The universal $1$-singular  Gelfand-Tsetlin modules have finite length and most of their irreducible subquotients are  examples of irreducible singular Gelfand-Tsetlin $\gl(n)$-modules that are previously not known. We also obtain a condition for the universal module to be irreducible.

(2) We  show that for any $1$-singular Gelfand-Tsetlin character there  exist at most two non-isomorphic irreducible Gelfand-Tsetlin modules with that character, and that the corresponding Gelfand-Tsetlin multiplicity of this character in any of the two modules is at most $2$. We prove that, except for one case, every irreducible $1$-singular Gelfand-Tsetlin module is a subquotient of a universal derivative tableaux Gelfand-Tsetlin module. This single case occurs when the universal module has two isomorphic irreducible subquotients. We conjecture that even in that single case the realization of the irreducibles as subquotients of universal derivative tableaux modules remain valid. This conjecture is true in the case of $\gl (3)$  as we show in the forthcoming paper \cite{FGR}. In particular, with the aid of derivative tableaux, we obtain a classification and explicit tableaux realization of all irreducible Gelfand-Tsetlin $\gl (3)$-modules.

The paper is organized as follows. In Section 3 we recall the classical construction of Gelfand and Tsetlin and collect important properties for the generic Gelfand-Tsetlin modules. In this section we also rewrite the Gelfand-Tsetlin formulas in terms of permutations. In Section 4 we introduce our new construction - the derivative tableaux and prove that the space $V(T(\bar{v}))$ spanned by the (usual) tableaux $T(\bar{v} + z)$ and the derivative tableaux ${\mathcal D} T (\bar{v} + z)$ associated to a $1$-singular vector $\bar{v}$  has a $\gl (n)$-module structure. The proof that this $\gl (n)$-module is a 1-singular Gelfand-Tsetlin module is included in Section 5, where explicit formulas for the action of the Gelfand-Tsetlin subalgebra on $V(T(\bar{v}))$ are obtained. In Section 6 we obtain a condition for the irreducibility of $V(T(\bar{v}))$. In Section 7 we show that there are at most two Gelfand-Tsetlin modules associated with a fixed 1-singular Gelfand-Tsetlin character and show that, except for one case, every irreducible $1$-singular Gelfand-Tsetlin module is a subquotient of $V(T(\bar{v}))$ for some $\bar{v}$. In the last section, the appendix, we prove some technical results that are needed in Section 4. 

\noindent{\bf Acknowledgements.} D.G. gratefully acknowledges the
hospitality and excellent working conditions at the S\~ao Paulo
University and  IH\'ES, where part of this work was completed. V.F. is
supported in part by  CNPq grant (301320/2013-6) and by 
Fapesp grant ( 2014/09310-5 ). D.G is supported in part by 
Fapesp grant \\(2011/21621-8) and by NSA grant H98230-13-1-0245. L.E.R. is supported by 
Fapesp grant (2012/23450-9).

\section{Conventions and notation} The ground field will be ${\mathbb C}$.  For $a \in {\mathbb Z}$, we write $\mathbb Z_{\geq a}$ for the set of all integers $m$ such that $m \geq a$.  We fix an integer $n\geq 2$. By $\gl(n)$ we denote the general linear Lie
algebra consisting of all $n\times n$ complex matrices, and by $\{E_{i,j}\mid 1\leq i,j \leq n\}$  - the
standard basis of $\gl(n)$ of elementary matrices. We fix the standard triangular decomposition and  the corresponding basis of simple roots of $\gl(n)$.  The weights of $\gl(n)$ will be written as $n$-tuples $(\lambda_1,...,\lambda_n).$ 

For a Lie algebra ${\mathfrak a}$ by $U(\mathfrak a)$ we denote the universal enveloping algebra of ${\mathfrak a}$. Throughout the paper $U = U(\gl (n))$.  For a commutative ring $R$, by ${\rm Specm}\, R$ we denote the set of maximal ideals of $R$.

We will write the vectors in $\mathbb{C}^{\frac{n(n+1)}{2}}$ in the following form:
$$
z=(z_{n1},...,z_{nn}|z_{n-1,1},...,z_{n-1,n-1}| \cdots|z_{21}, z_{22}|z_{11}).
$$
For $1\leq j \leq i \leq n$, $\delta^{ij} \in {\mathbb Z}^{\frac{n(n+1)}{2}}$ is defined by  $(\delta^{ij})_{ij}=1$ and all other $(\delta^{ij})_{k\ell}$ are zero. 
 
For $i>0$ by $S_i$ we denote the $i$th symmetric group. By $(k, \ell)$ we denote the transposition of $S_n$ switching $k$ and $\ell$. Throughout the paper we set  $G:=S_n\times\cdots \times S_1$ and  $Q_n:=1!2!\ldots (n-1)!$ Every element $\sigma$ in $G$ will be written as an $n$-tuple $(\sigma[n],...,\sigma[1])$ for permutations $\sigma[i] \in S_i$.

\section{Gelfand-Tsetlin modules}

\subsection{Definitions}
Recall that $U=U(\gl (n) )$.  Let  for $m\leqslant n$, $\mathfrak{gl}_{m}$ be the Lie subalgebra
of $\gl (n)$ spanned by $\{ E_{ij}\,|\, i,j=1,\ldots,m \}$. Then we have the following chain
$$\gl_1\subset \gl_2\subset \ldots \subset \gl_n,$$
which induces  the chain $U_1\subset$ $U_2\subset$ $\ldots$ $\subset
U_n$ of the universal enveloping algebras  $U_{m}=U(\gl_{m})$, $1\leq m\leq n$. Let
$Z_{m}$ be the center of $U_{m}$. Then $Z_m$ is the polynomial
algebra in the $m$ variables $\{ c_{mk}\,|\,k=1,\ldots,m \}$,
\begin{equation}\label{equ_3}
c_{mk } \ = \ \displaystyle {\sum_{(i_1,\ldots,i_k)\in \{
1,\ldots,m \}^k}} E_{i_1 i_2}E_{i_2 i_3}\ldots E_{i_k i_1}.
\end{equation}

Following \cite{DFO3}, we call the subalgebra of $U$ generated by $\{
Z_m\,|\,m=1,\ldots, n \}$ the \emph{(standard) Gelfand-Tsetlin
subalgebra} of $U$ and will be denoted by  ${\Ga}$. In fact,  ${\Ga}$ is the polynomial algebra in the $\displaystyle \frac{n(n+1)}{2}$ variables $\{
c_{ij}\,|\, 1\leqslant j\leqslant i\leqslant n \}$ (\cite{Zh}).
 Let $\Lambda$ be the polynomial
algebra in the variables $\{\lambda_{ij}\,|$ $1\leqslant j\leqslant
i\leqslant n \}$.

Let $\imath:{\Ga}{\longrightarrow}$ $\Lambda$ be the embedding
defined by $\imath(c_{mk}) \  = \  \gamma_{mk} (\lambda)$, where
\begin{equation} \label{def-gamma}
\gamma_{mk} (\lambda): = \ \sum_{i=1}^m
(\lambda_{mi}+m-1)^k \prod_{j\ne i} \left( 1 -
\frac{1}{\lambda_{mi}-\lambda_{mj}} \right).
\end{equation} The image of $\imath$
coincides with the subalgebra of $G$-invariant polynomials  in
$\Lambda$ (\cite{Zh}) which we identify with $\Ga$. Note that $\Lambda$ is the integral closure of $\Ga$ in $L$. Then we have a surjective map  $\pi: \Sp \Lambda \rightarrow 
\Sp \Ga$.
If $\pi(\ell)=\sm$ for some
$\ell\in \Specm \Lambda$, then
we write $\ell=\ell_{\sm}$ and say that $\ell_{\sm}$ is {\it lying over} $\sm$.

 Denote by $K$ be
the field of fractions of ${\Gamma}$.
Let $L$ be the fraction field
of $\Lambda$. Let $\mathcal M $ be the free abelian group generated by
$\delta^{ij}$, $1\leqslant j\leqslant i\leqslant n-1$. Clearly, $\Sp \Lambda \simeq {\mathbb C}^{\frac{n(n+1)}{2}}$ and $\mathcal M \simeq \mathbb
Z^{\frac{n(n-1)}{2}}$.
Then $G$ acts naturally on $\Sp \Lambda$ by conjugation and  $\mathcal M$ acts on $\Sp \Lambda$ by the corresponding shift: $\delta^{ij}\cdot\ell=$ $\ell+\delta^{ij}$,
$\delta^{ij}\in$ $\mathcal M$. We have $L^G=K$, $\Lambda^{G}=\Ga$ and $G=G(L/K)$ is the
Galois group of the field extension $ K\subset L$.

\begin{definition}
\label{definition-of-GZ-modules} A finitely generated $U$-module
$M$ is called a \emph{Gelfand-Tsetlin module (with respect to
$\Ga$)} if $M$ splits into  a direct sum
of $\Ga$-modules:

\begin{equation*}
M=\bigoplus_{\sm\in\Sp\Ga}M(\sm),
\end{equation*}
where $$M(\sm)=\{v\in M| \sm^{k}v=0 \text{ for some }k\geq 0\}.$$
\end{definition}

Identifying $\sm$ with the homomorphism $\chi:\Gamma \rightarrow {\mathbb C}$ with $\Ker \chi=\sm$, we will call $\sm$ 
a \emph{Gelfand-Tsetlin character} of $M$ if $ M(\sm) \neq 0$, and $\dim M(\sm)$ - the \emph{Gelfand-Tsetlin multiplicity of $\sm$}. The  \emph{Gelfand-Tsetlin support} of a Gelfand-Tsetlin module $M$ is the set of all Gelfand-Tsetlin characters of $M$. We will often call  Gelfand-Tsetlin character, Gelfand-Tsetlin multiplicity, and Gelfand-Tsetlin support simply  character, multiplicity, and support, respectively.

Note that any irreducible Gelfand-Tsetlin  module over $\gl(n)$ is a weight module with respect to the
 standard Cartan subalgebra $\mathfrak h$ spanned by $E_{ii}$, $i=1,\ldots, n$. The converse is  not true in general (except for $n=2$), i.e. an irreducible weight module $M$ need not to be Gelfand-Tsetlin. However, it is the case when
the weight multiplicities of $M$ are finite. In particular, every highest weight module or, more generally, every module from the category $\mathcal O$ is a Gelfand-Tsetlin module.

\subsection{Preliminaries}
In this subsection we collect some general results on Gelfand-Tsetlin modules. The first two theorems follow from the fact that $U$ is a Galois order with respect to $\Ga$, \cite{FO1}. 
Another important fact used in the proofs is  that $U$ is free as a left and as a right $\Ga$-module, \cite{Ovs}. 

The inequality in the next theorem follows by Lemma 4.1(c) and Theorem 4.12(c) in \cite{FO2}. Note that, as proved in \cite[\S 4]{FO2}, the set on the right hand side of the inequality does not depend on the choice of $\ell_{\sm}$.
\begin{theorem}\label{thm-bound-dim} 
Let  $M\neq 0$ be a Gelfand-Tsetlin $\gl(n)$-module,  and let $\sm\in\Specm \Ga$ be in the support of $M$ such that $M$ is generated by 
  $x\in M(\sm)$ and $\sm x = 0$.  Then, for each $\sn$ in the support of $M$
    \begin{equation*}
    \dim M(\sn) \leq |\{\phi\in \cM\mid \pi(\phi\ell_{\sm})=\sn\}|.
    \end{equation*}
 
\end{theorem}

Recall $Q_n=1!2!\ldots (n-1)!$.

\begin{theorem}[\cite{FO2}, Corollary 5.3]
\label{theorem-finiteness-for-gl-n} Let $\sm\in\Sp\Ga$. Then
\begin{itemize}
\item[(i)]\label{enumerate-the-dimension-gl-simples} If $M$ is a $U$-module generated by some $x\in M(\sm)$ (in particular for an irreducible module), then $$\dim
M(\sm)\leq Q_n.$$ 
\item[(ii)]\label{enumerate-the-number-gl-simples} The
number of  isomorphism classes of irreducible $U$-modules $N$ such that $N(\sm)\ne 0$ is always nonzero and does not exceed  $Q_n$.
\end{itemize}
\end{theorem}

The theorem above shows that the elements of $\Sp\Ga$ classify the irreducible $\gl(n)$-modules (and, hence, irreducible $\mathfrak{sl}(n)$-modules) up to some finiteness.

The following  result  will be used in Sections \ref{sec-action} and \ref{sec-last}.

\begin{lemma}\label{lem-cyclic-Gelfand-Tsetlin}
Let $\sm\in\Sp\Ga$ and $M$  be  a $\gl(n)$-module generated by a nonzero element $v\in M(\sm)$. Then $M$ is a Gelfand-Tsetlin module.  

\end{lemma} 

\begin{proof}
Recall that a commutative subalgebra $A$ of some associative algebra $B$ is a Harish-Chandra subalgebra, if   for
any $b\in B$, the $A-$bimodule $A bA$ is
finitely generated both as a left and as a right
$A$-module (for details see \cite{DFO3}). By Corollary 5.4 and Proposition 7.2 in \cite{FO1}, the Gelfand-Tsetlin subalgebra $\Gamma$ is a Harish-Chandra subalgebra of $U$.

Let  $z_1,\ldots,z_m$, $m=\frac{n(n+1)}{2}$, be a set of generators of $\Gamma$. Then there exist polynomials $f_j\in {\mathbb C}[x]$, $j=1,\ldots,m$ such that $f_j(z_j)v=0$ for all $j$. Let $z$ be any generator of $\Gamma$. 
Since $\Gamma$ is a Harish-Chandra subalgebra, there exist $\gamma_1,\ldots,\gamma_s\in \Gamma$ such that for any large power $N$ of $z$ we can write
$$z^{N}E_{i,i+1}=\sum_{j=1}^s \gamma_i E_{i,i+1} g_j,$$
for any $i$ and for some $g_j\in \Gamma$.
  Using this  and the fact  that there are only finitely many linearly independent elements in the set 
 $$\{z_i^{t_i}v| t_i\in \mathbb Z_{\geq 0}, i=1, \ldots, m\},$$
 we conclude that $h_i(z)E_{i,i+1}v=0$ for each $i$ and some polynomial $h_i\in \mathbb C[x]$. We apply similar reasoning  for $E_{i,i-1}v$. The computation for $E_{ii}v$ is trivial since $E_{ii}\in \Gamma$, which completes the proof.  \end{proof}

\subsection{Finite dimensional modules for $\mathfrak{gl}(n)$}

In this section we recall a classical result of  Gelfand and  Tsetlin which provides an explicit basis  for every irreducible finite dimensional $\mathfrak{gl}(n)$-module.

\begin{definition} The following array  $[T]$ with complex  entries $\{t_{ij}:1\leq j\leq i\leq n\}$ 
\medskip
\begin{center}

\Stone{\mbox{ \scriptsize {$t_{n1}$}}}\Stone{\mbox{ \scriptsize {$t_{n2}$}}}\hspace{1cm} $\cdots$ \hspace{1cm} \Stone{\mbox{ \scriptsize {$t_{n,n-1}$}}}\Stone{\mbox{ \scriptsize {$t_{nn}$}}}\\[0.2pt]
\Stone{\mbox{ \scriptsize {$t_{n-1,1}$}}}\hspace{1.5cm} $\cdots$ \hspace{1.5cm} \Stone{\mbox{ \tiny {$t_{n-1,n-1}$}}}\\[0.3cm]
\hspace{0.2cm}$\cdots$ \hspace{0.8cm} $\cdots$ \hspace{0.8cm} $\cdots$\\[0.3cm]
\Stone{\mbox{ \scriptsize {$t_{21}$}}}\Stone{\mbox{ \scriptsize {$t_{22}$}}}\\[0.2pt]
\Stone{\mbox{ \scriptsize {$t_{11}$}}}\\
\medskip
\end{center}
is called a \emph{Gelfand-Tsetlin tableau}. 
\end{definition}

We will often identify $[T]$ with an element of ${\mathbb C}^{\frac{n(n+1)}{2}} \simeq \Sp \Lambda$ and  write $[T] = (t_{n1},...,t_{nn}|...|t_{21}, t_{22}|t_{11})$ or simply $[T] = (t_{ij})$. In particular we consider $\delta^{ij}$ and $\ell_{\sm}$ defined in the previous section as Gelfand-Tsetlin tableaux.

A Gelfand-Tsetlin tableau  $(t_{ij})$  is called \emph{standard} if:
$$t_{ki}-t_{k-1,i}\in\mathbb{Z}_{\geq 0} \hspace{0.3cm} and \hspace{0.3cm} t_{k-1,i}-t_{k,i+1}\in\mathbb{Z}_{> 0}, \hspace{0.2cm}\text{ for all } 1\leq i\leq k\leq n-1.$$
Note that, for sake of convenience, the second condition above is slightly different from the original condition in \cite{GT}.

\begin{theorem}[\cite{GT}, \cite{Zh}]\label{Gelfand-Tsetlin theorem}
Let $L(\lambda)$ be the finite dimensional irreducible module over $\mathfrak{gl}(n)$ of highest weight $\lambda=(\lambda_{1},\ldots,\lambda_{n})$. Then there exists a basis of $L(\lambda)$ consisting of all standard tableaux $[T] = (t_{n1},...,t_{nn}|...|t_{21}, t_{22}|t_{11})$ with fixed top row $t_{n1}=\lambda_1,t_{n2}=\lambda_2-1,\ldots,t_{nn}=\lambda_{n}-n+1$. Moreover,  the action of the generators of $\mathfrak{gl}(n)$ on $L(\lambda)$ is given by the  \emph{Gelfand-Tsetlin formulas}:

$$E_{k,k+1}([T])=-\sum_{i=1}^{k}\left(\frac{\prod_{j=1}^{k+1}(t_{ki}-t_{k+1,j})}{\prod_{j\neq i}^{k}(t_{ki}-t_{kj})}\right)[T+\delta^{ki}],$$

$$E_{k+1,k}([T])=\sum_{i=1}^{k}\left(\frac{\prod_{j=1}^{k-1}(t_{ki}-t_{k-1,j})}{\prod_{j\neq i}^{k}(t_{ki}-t_{kj})}\right)[T-\delta^{ki}],$$

$$E_{kk}([T])=\left(\sum_{i=1}^{k}(t_{ki}+i-1)-\sum_{i=1}^{k-1}(t_{k-1,i}+i-1)\right)[T],$$
where $[T\pm\delta^{ki}]$ is the tableau obtained by $[T]$ adding $\pm 1$ to the $(k,i)$-th entry of $[T]$. If the new tableau $[T\pm\delta^{ki}]$ is not standard, then the corresponding summand of $E_{k,k+1}([T])$ or $E_{k+1,k}([T])$ is zero by definition.  Furthermore, for $s\leq r$, 
\begin{equation}\label{action of Gamma in finite dimensional modules}
c_{rs}([T])=\gamma_{rs}(t)[T],
\end{equation}
where $\gamma_{rs}$ are defined in (\ref{def-gamma}).
\end{theorem}

\noindent

One immediate consequence of the above theorem is that the algebra $\Gamma$ acts semisimply on  any finite dimensional irreducible module $L(\lambda)$. Moreover, $L(\lambda)$ has a \emph{simple spectrum}, that is, all Gelfand-Tsetlin multiplicities are $1$.

\subsection{Generic Gelfand-Tsetlin modules}\label{section generic modules}

Observing that the coefficients in the Gelfand-Tsetlin formulas in Theorem \ref{Gelfand-Tsetlin theorem} are rational functions on the entries of the tableaux, it is natural to extend the Gelfand-Tsetlin construction to more general modules.  In the case when all denominators are nonintegers, one can use the same formulas and  define a new class of infinite dimensional $\gl(n)$-modules:  {\it generic}  Gelfand-Tsetlin modules (cf. \cite{DFO3}, Section 2.3).
\begin{definition}\label{generic tableau definition}
A Gelfand-Tsetlin tableau $[T]=(t_{ij})$ is called \emph{generic} if $t_{rs}-t_{ru}\notin\mathbb{Z}$ for each $1\leq s<u\leq r\leq n-1$. 
\end{definition}
\begin{theorem}[\cite{DFO3}, Section 2.3]\label{Generic Gelfand-Tsetlin modules}
Let $[T] = (t_{ij})$ be a generic  Gelfand-Tsetlin tableau. Denote by $V([T])$  the set of all Gelfand-Tsetlin tableaux $[L] = (l_{ij})$ satisfying $l_{nj}=t_{nj}$, $l_{ij}-t_{ij}\in\mathbb{Z}$ for , $1\leq j\leq i \leq n$. 
\begin{itemize}
\item[(i)] The vector space with basis $V([T])$ has  a structure of a $\mathfrak{gl}(n)$-module with action of the generators of $\mathfrak{gl}(n)$ given by the Gelfand-Tsetlin formulas.
\item[(ii)] The action of the generators of $\Gamma$ on the basis elements of $V([T])$ is given by (\ref{action of Gamma in finite dimensional modules}). 
\item[(iii)] The module defined  in {\rm (i)} is a Gelfand-Tsetlin module all of whose Gelfand-Tsetlin multiplicities are $1$.
\end{itemize}
\end{theorem}
Note that since $[T]$ is generic all denominators of $E_{k,k+1}([T])$ and $E_{k+1,k}([T])$, $[T] \in V([L])$, are nonzero, so the condition that the summands corresponding to nonstandard tableaux are zero in Theorem \ref{Generic Gelfand-Tsetlin modules} is not needed.  By a slight abuse of notation we will  denote  the module constructed in  Theorem \ref{Generic Gelfand-Tsetlin modules} by $V([T])$ and will call it \emph{the universal generic Gelfand-Tsetlin  module associated with $[T]$}. Note that $V([T])$ need not to be irreducible.  Because $\Gamma$ has simple spectrum on $V([T])$ for $[L]$ in $V([T])$ we may define the \emph{irreducible $\mathfrak{gl}(n)$-module in $V([T])$ containing $[L]$} to be the subquotient of $V([T])$ containing $[L]$ (see Remark \ref{V([T]) has finite length} below).
A basis for the irreducible subquotients of $V([T])$  can be described as follows.

\begin{definition}\label{definition of Omega +} Let $[T] = (t_{ij})$ be a fixed Gelfand-Tsetlin tableau. For any $[L]=(l_{ij})\in V([T])$, and for any $1<r\leq n$, $1\leq s\leq r$ and $1\leq u \leq r-1$ we define:
$$\Omega^{+}([L]):=\{(r,s,u): l_{rs}-l_{r-1,u}\in \mathbb{Z}_{\geq 0}\}$$
\end{definition}

The following theorem is a subject of a direct verification. Details can be found in \cite{FGR2}

\begin{theorem}\label{Basis for irreducible generic modules gl(n)}
Let $[T]$ be a generic tableau and let $[L]$ be a tableau in $V([T])$. Then the following hold.
\begin{itemize}
\item[(i)] The submodule of $V([T])$ generated by $[L]$  has basis $$\mathcal{N}([L]):=\{[R]\in V([T]):\Omega^{+}([L])\subseteq \Omega^{+}([R])\};$$
\item[(ii)] The irreducible $\mathfrak{gl}(n)$-module in $V([T])$ containing $[L]$ has basis
$$\mathcal{I}([L]):=\{[R]\in V([T]):\Omega^{+}([L])=\Omega^{+}([R])\}.$$
\end{itemize}
The action of $\mathfrak{gl}(n)$ on both $\mathcal{N}([L])$ and $\mathcal{I}([L])$ is given by the Gelfand-Tsetlin formulas. 
\end{theorem}

\begin{remark}\label{V([T]) has finite length}
The module $V([T])$ has finite length. Indeed,  $V([T])$ is generated by any tableau $[L]\in V([T])$ such that $\Omega^{+}([L])=\emptyset$. Then, since $V([T])$ is a cyclic module, it is a quotient of $U/U\sm$ for some generic $\sm$. By Theorem $4.14$ in \cite{FO2}, the module $U/U\sm$ (and hence $V([T])$ itself) has finite length.
\end{remark}

\subsection{ Gelfand-Tsetlin formulas in terms of permutations}
\medskip

In this subsection we rewrite and generalize the Gelfand-Tsetlin formulas in Theorem \ref{Gelfand-Tsetlin theorem} in convenient for us terms.  

For a vector  $v=(v_{n1},...,v_{nn}|\cdots|v_{11})$ in $\mathbb{C}^{\frac{n(n+1)}{2}}$, by $[T(v)]$, or simply by $T(v)$, we will denote the corresponding to $v$ Gelfand-Tsetlin tableau.  Let us call $v$ in ${\mathbb C}^{\frac{n(n+1)}{2}}$ {\it generic} if $T(v)$ is a generic Gelfand-Tsetlin tableau, and denote by ${\mathbb C}^{\frac{n(n+1)}{2}}_{\rm gen}$ the set of all generic vectors in ${\mathbb C}^{\frac{n(n+1)}{2}}$.

Let $\widetilde{S}_m$ denotes the subset of $S_m$ consisting of the transpositions $(1,i)$, $i=1,...,m$. For $\ell < m $, set $\Phi_{\ell m} =  \widetilde{S}_{m-1} \times\cdots\times \widetilde{S}_{\ell}$. For $\ell > m$ we set $\Phi_{\ell m} = \Phi_{m \ell}$. Finally we let $\Phi_{\ell \ell} = \{ \mbox{Id}\}$.  Every $\sigma$ in $\Phi_{\ell m }$ will be written as a $|\ell- m|$-tuple  of transpositions $\sigma[t]$ (recall that $\sigma[t]$ is the $t$-th component of $\sigma$). Also, we consider every $\sigma\in \Phi_{\ell m} $ as an element of $G = S_n \times \cdots \times S_1$ by letting  $\sigma [t] = \mbox{Id}$ whenever  $t<\min(\ell, m)$ or  $ t > \max(\ell, m)-1$.

\begin{remark}
Recall that in order to have well defined action of $\Phi_{\ell m }$ on $\mathbb{C}^{\frac{n(n+1)}{2}}$, for $w\in\mathbb{C}^{\frac{n(n+1)}{2}}$ and $\sigma\in\Phi_{\ell m }$, we set  
$$
\sigma(w):=(w_{n,\sigma^{-1}[n](1)},\ldots,w_{n,\sigma^{-1}[n](1)}|\ldots|w_{1,\sigma^{-1}[1](1)}).$$
\end{remark}

\begin{definition} \label{def-e-lm}
Let $1 \leq r < s \leq n$. Set
$$
\varepsilon_{rs}:=\delta^{r,1}+\delta^{r+1,1}+\ldots+\delta^{s-1,1} \in \mathbb{C}^{\frac{n(n+1)}{2}}.
$$
Furthermore, define $\varepsilon_{rr}=0$ and  $\varepsilon_{sr}=- \varepsilon_{rs}$.
\end{definition}

\begin{definition}\label{definition of coefficients e_rs}
For each generic vector $w$ and any $1\leq r, s\leq n$ we define
$$
e_{rs}(w):=
\begin{cases}
\left(\displaystyle\prod_{j=r}^{s-2}e_{j}^{(+)}(w)\right)e_{s-1,s}(w) & \text { if }\ \ \ r<s\\
e_{s+1,s}(w)\left(\displaystyle\prod_{j=s+2}^{r}e_{j}^{(-)}(w)\right) &  \text { if }\ \ \ r>s\\
\sum_{i=1}^{r}(w_{ri}+i-1)-\sum_{i=1}^{r-1}(w_{r-1,i}+i-1) &  \text { if }\ \ \ r=s,
\end{cases}
$$
where
\begin{align*}
e_{t}^{(+)}(w):=\frac{\prod_{j\neq 1}^{t+1}(w_{t1}-w_{t+1,j})}{\prod_{j\neq 1}^{t}(w_{t1}-w_{tj})} & \ \ \ \ ; \ \ \ \
e_{t+1}^{(-)}(w) :=\frac{\prod_{j\neq 1}^{t-1}(w_{t1}-w_{t-1,j})}{\prod_{j\neq 1}^{t}(w_{t1}-w_{tj})} \\
e_{k,k+1}(w):=- \frac{\prod_{j=1}^{k+1}(w_{k1}-w_{k+1,j})}{\prod_{j\neq 1}^{k}(w_{k1}-w_{kj})} & \ \ \ \ ; \ \ \ \ e_{k+1,k}(w):= \frac{\prod_{j=1}^{k-1}(w_{k1}-w_{k-1,j})}{\prod_{j\neq 1}^{k}(w_{k1}-w_{kj})}.
\end{align*}
\end{definition}

It is not difficult to prove the following generic module version of Theorem \ref{Gelfand-Tsetlin theorem} (see also Theorem 2 in \cite{Maz2}).

\begin{proposition}\label{coefficients e_ij} Let $v \in {\mathbb C_{\rm gen}^{\frac{n(n+1)}{2}}}$. The Gelfand-Tsetlin formulas for the generic Gelfand Tsetlin $\mathfrak{gl}(n)$-module  $V(T(v))$ can be written as follows:
$$
E_{\ell m} (T(v+z))= \sum_{\sigma \in \Phi_{\ell m}} e_{\ell m} (\sigma (v+z)) T(v+z+\sigma(\varepsilon_{\ell m})),$$
for $z \in {\mathbb Z^{\frac{n(n+1)}{2}}}$ and $1 \leq m \leq \ell \leq n$. 
\end{proposition}

\section{Derivative Tableaux} \label{sec-der}
In this section we consider tableaux $T(v)$ in $\mathbb{C}^{\frac{n(n+1)}{2}}$ with fixed first row $v_{nm}=a_m$, $m=1,...,n$. For this reason we will assume that the corresponding vectors $v$ are in $\mathbb{C}^{\frac{n(n-1)}{2}}$. The goal is to introduce a new notion of tableaux, called derivative tableaux, and to define a module structure on the space spanned by all regular and derivative tableaux.

\subsection{Singular tableaux}

\begin{definition}
A vector $v\in\mathbb{C}^{\frac{n(n-1)}{2}}$ will be called \emph{singular} if there exist $1\leq s<t\leq r\leq n-1$ such that $v_{rs}-v_{rt}\in\mathbb{Z}$. The vector $v$ will be called \emph{$m$-singular} if for $i=1,...,m$, there exist $r_{i},s_{i},t_{i}$ with $1\leq s_{i}<t_{i}\leq r_{i}\leq n-1$ such that $v_{r_{i}s_{i}}-v_{r_{i}t_{i}}\in\mathbb{Z}$ for all $i=1,\ldots,m$ and $v_{rs}-v_{rt}\notin\mathbb{Z}$ for all $(r,s,t) \neq (r_{i},s_{i},t_{i})$, $i=1,\ldots,m$. By defintion, every generic vector is $0$-singular.
\end{definition}

{\it From now on we fix $(i,j,k)$ such that  $1\leq i < j \leq k\leq n-1$}. By ${\mathcal H} = {\mathcal H}_{ij}^k$ we denote the hyperplane $v_{ki} - v_{kj} = 0$ in ${\mathbb C}^{\frac{n(n-1)}{2}}$ and let $\overline{\mathcal H} = \overline{\mathcal H}_{ij}^k$ be the subset of all $v$ in ${\mathbb C}^{\frac{n(n-1)}{2}}$ such that $v_{tr} \neq v_{ts}$ for all triples $(t,r,s)$ except for $(t,r,s) = (k,i,j)$. {\it From now on we fix $\bar{v}$ in ${\mathcal H}$ such that $\bar{v}_{ki} = \bar{v}_{kj} $ and all other differences $v_{mr} - v_{ms}$ are noninteger.}  In other words, $\bar{v} \in {\mathcal H}$ and $\bar{v} + {\mathbb Z}^{\frac{n(n-1)}{2}} \subset \overline{\mathcal H}$. In particular, $\bar{v}$ is a $1$-singular vector. A character $\chi$ and $\sn = \Ker \chi$ are called {\it 1-singular} if $\ell_{\sn}$ is $1$-singular for one choice (hence for all choices) of $\ell_{\sn}$. A Gelfand-Tsetlin module $M$ will be called {\it $1$-singular Gelfand-Tsetlin module} if $M(\sn) \neq 0$ for some $1$-singular $\sn \in \Specm \Gamma$.

 Our goal is to define a module $V(T(\bar{v}))$ whose support contains the set of tableaux $\left\{T(\bar{v}+ w)\; | \; w \in{\mathbb Z}^{\frac{n(n-1)}{2}}\right\}$. This module will be our {\it  universal $1$-singular Gelfand-Tsetlin module}. In order to define  $V(T(\bar{v}))$ we first introduce a family ${\mathcal V}_{\rm gen}$ of generic Gelfand-Tsetlin modules as follows.  
 
 Since for a generic $v$, $V(T(v)) = V(T(v'))$ whenever $v - v' \in {\mathbb Z}^{\frac{n(n-1)}{2}}$, we may define $V(T(v))$ for $v$ in the (generic) complex torus $T= {\mathbb C}^{\frac{n(n-1)}{2}}_{\rm gen} / {\mathbb Z}^{\frac{n(n-1)}{2}}$. We will take a direct sum of such generic $V(T(v))$ by choosing representatives $w+ {\mathbb Z}^{\frac{n(n-1)}{2}}$ of $w$ in $T$ as ``close'' as possible to $\bar{v}$ as follows.

 \begin{definition} \label{def-s}
For  $w \in {\mathbb C}^{\frac{n(n-1)}{2}}_{\rm gen}$, let $\lfloor Re(\bar{v} -w )\rfloor$ be the vector in ${\mathbb C}^{\frac{n(n-1)}{2}}$ whose $(r,s)$th component is $\lfloor Re(\bar{v}_{rs}-w_{rs})\rfloor$ (the integer part of the real part of $\bar{v}_{rs}-w_{rs}$).  We set
$$\mathcal{S}:=\left\{w+\lfloor Re(\bar{v} -w)\rfloor \; | \; w\in{\mathbb C}^{\frac{n(n-1)}{2}}_{\rm gen}\right\}.$$
\end{definition}

\begin{remark} The elements in $\mathcal S$ are as close as possible to $\bar{v}$ in the following sense. If $\bar{v}[w] :=w+\lfloor Re(\bar{v} -w)\rfloor$, then $\lfloor Re(\bar{v}_{rs}-\bar{v}[ w ]_{rs})\rfloor =0$ for any $r,s$. Actually, if $u \in w + {\mathbb Z}^{\frac{n(n-1)}{2}}$, and $\lfloor Re(v_{rs}-u_{rs})\rfloor =0$ for any $r,s$, then we have $u=\bar{v}[w ]$.
\end{remark}

Now define the generic family of modules ${\mathcal V}_{\rm gen} := \bigoplus_{v \in \mathcal{S}} V(T(v))$.  

Denote by ${\mathcal F}$ the space of rational functions on $v_{\ell m}$, $1\leq m \leq \ell \leq n$, with poles on the hyperplanes $v_{rs} - v_{rt} =0$, and by ${\mathcal F}_{ij}$  the subspace of ${\mathcal F}$ consisting of all functions that are smooth on $\overline{\mathcal H}$. Then ${\mathcal F} \otimes {\mathcal V}_{\rm gen}$  is a  $\mathfrak{gl} (n)$-module  with the  trivial action on ${\mathcal F}$.

 \subsection{The space of derivative tableaux}
{\it From now on by $\tau$ we denote  the element in $S_{n-1} \times\cdots \times S_{1}$ such that $\tau[k]$ is the transposition $(i,j)$  and all other $\tau[t]$ are $\mbox{Id}$.} In particular $w \in {\mathcal H}$ if and only if $\tau(w) = w$. 
 
Since $\bar{v}$ is a $1$-singular vector, the generators of $\Gamma$ acts in the same way on $T(\bar{v}+z)$ and $T(\bar{v}+\tau(z))$, with some abuse of notation we will write $T(\bar{v} + z) - T(\bar{v} +\tau(z)) = 0$. We first introduce formally new tableaux  ${\mathcal D}_{ij} T({\bar{v}} + z)$ for  every $z \in {\mathbb Z}^{\frac{n(n-1)}{2}}$ subject to the relations ${\mathcal D}_{ij} T({\bar{v}} + z) + {\mathcal D}_{ij} T({\bar{v}} + \tau(z)) = 0$ or, equivalently, ${\mathcal D}_{ij} T(u) + {\mathcal D}_{ij} T(\tau (u)) = 0$ for all $u$ in ${\bar{v}} +  {\mathbb Z}^{\frac{n(n-1)}{2}}$. We call ${\mathcal D}_{ij} T(u)$  {\it the derivative Gelfand-Tsetlin tableau} associated with $u$.  

Now, define  $V(T(\bar{v}))$ to be the space spanned by $\{ T(\bar{v} + z), \,\mathcal{D}_{ij} T(\bar{v} + z) \; | \; z \in {\mathbb Z}^{\frac{n(n-1)}{2}}\}$. Note that this set is not a basis since $T(\bar{v} + z) - T(\bar{v} +\tau(z)) = 0$, ${\mathcal D}_{ij} T(\bar{v} + z) + {\mathcal D}_{ij} T(\bar{v} +\tau(z)) = 0$. A basis of $V(T(\bar{v}))$ is for example the set
$$
\{ T(\bar{v} + z), \mathcal{D}_{ij} T(\bar{v} + w) \; | \; z_{ki} \leq z_{kj}, w_{ki} > w_{kj}\}.
$$

Set ${\mathcal V}' = V(T(\bar{v})) \oplus {\mathcal V}_{\rm gen}$.  Define {\it the evaluation map} $\mbox{ev}(\bar{v}) : {\mathcal F}_{ij} \otimes {\mathcal V}' \to  {\mathcal V}'  $, which is linear and 
$$
f T(v+z) \mapsto f(\bar{v}) T(\bar{v}+z), \, f {\mathcal D}_{ij} T(\bar{v}+z) \mapsto f(\bar{v}) {\mathcal D}_{ij} T(\bar{v}+z),$$  for $z \in {\mathbb Z}^{\frac{n(n-1)}{2}}$, $f\in {\mathcal F}_{ij}$ and $v \in {\mathcal S}$.


Finally, let $\mathcal{D}_{ij}^{\bar{v}}:  {\mathcal F}_{ij} \otimes V(T(v)) \to  V(T(\bar{v}))   $ be the linear map defined by 
$$
\mathcal{D}_{ij}^{\bar{v}} (f T(v+z)) = \mathcal{D}_{ij}^{\bar{v}} (f) T(\bar{v}+z) +   f(\bar{v}) \mathcal{D}_{ij} T(\bar{v}+z),
$$ 
where $\mathcal{D}_{ij}^{\bar{v}}(f) = \frac{1}{2}\left(\frac{\partial f}{\partial v_{ki}}-\frac{\partial f}{\partial v_{kj}}\right)(\bar{v})$. In other words, this is the map $\mathcal{D}_{ij}^{\bar{v}} \otimes \mbox{ev}(\bar{v})  + \mbox{ev}(\bar{v})  \otimes \mathcal{D}_{ij}^{\bar{v}}$. This map certainly extends to a linear map  ${\mathcal F}_{ij} \otimes {\mathcal V}_{\rm gen} \to V (T(\bar{v}))$ which we will also denote by $\mathcal{D}_{ij}^{\bar{v}}$. In particular, $\mathcal{D}_{ij}^{\bar{v}} (T(v+z)) = \mathcal{D}_{ij} (T(\bar{v}+z))$.

\begin{remark}
The operator $\mathcal{D}_{ij}^{\bar{v}}: {\mathcal V}_{\rm gen} \to V (T(\bar{v}))$  can be considered as a formal derivation. Namely, for $v \in {\mathbb C}^{\frac{n(n+1)}{2}}$ and $z \in {\mathbb Z}^{\frac{n(n-1)}{2}}$, ${\mathcal D}_{ij}^{\bar{v}} T(v + z)$ is the formal limit of  $\frac{T(v+z) - T(v+ \tau(z))}{v_{ki} - v_{kj}}$ when $v \to \bar{v}$.
\end{remark}

\subsection{Module structure on $V(T(\bar{v}))$} {\it From now on we set for convenience ${\mathcal D}^{\bar{v}} = {\mathcal D}_{ij}^{\bar{v}}$, ${\mathcal D}T(\bar{v}+z)={\mathcal D}_{ij}^{\bar{v}}(T(v+z)) = {\mathcal D}_{ij}T(\bar{v}+z)$ ,  $x = v_{ki}$, and $y = v_{kj}$.}

We define the action of $\mathfrak{gl}(n)$ on the generators of $V(T(\bar{v}))$ as follows:
\begin{align*}
E_{rs}(T(\bar{v} + z))=&\  \mathcal{D}^{\bar{v}}((x - y)E_{rs}(T(v + z)))\\
E_{rs}(\mathcal{D}T(\bar{v} + w)))=&\ \mathcal{D}^{\bar{v}} ( E_{rs}(T(v + w))),
\end{align*}
where $v$ is a generic vector, $z, w \in {\mathbb Z}^{\frac{n(n-1)}{2}}$ with $w \neq \tau(w)$. One should note that $(x - y)E_{mn}(T(v + z))$ and $E_{mn}(T(v + w))$ are in $ {\mathcal F}_{ij} \otimes V(T(v))$, so the right hand sides in the above formulas are well defined.

The following proposition is proved in \S \ref{proof-compatible}.
\begin{proposition} \label{compatible} Let $v$ be generic and $z \in {\mathbb Z}^{\frac{n(n-1)}{2}}$.
\begin{itemize}
\item[(i)] $ {\mathcal D}^{\bar{v}} ((x-y) E_{rs} T(v+z)) =  {\mathcal D}^{\bar{v}} ((x-y) E_{rs} T(v+\tau(z))) $ for all $z$.
\item[(ii)] $ {\mathcal D}^{\overline{v}}(E_{rs} T(v+z)) =  - {\mathcal D}^{\bar{v}} (E_{rs} T(v+\tau(z))) $ for all $z$ such that $\tau(z) \neq z$.
\end{itemize}
\end{proposition}

Now, with the aid of Proposition \ref{compatible} we define $E_{rs} F$ for any $F \in  {\mathcal F}_{ij} \otimes  V(T(\bar{v}))$. Then, using linearity we define $g F$ for any $g \in \mathfrak{gl} (n)$ and $F \in {\mathcal F}_{ij} \otimes  {\mathcal V}'$. It remains to prove that this well-defined action endows $ V(T(\bar{v}))$, and hence $ {\mathcal F}_{ij} \otimes  {\mathcal V}'$, with a $\mathfrak{gl} (n)$-module structure.

\begin{lemma} \label{dij-commute} Let $g \in \mathfrak{gl}(n)$.
\begin{itemize} 
\item[(i)] $g (T(\overline{v} + z)) = {\rm ev} ( \overline{v}) g ( T(v+z))$ whenever $\tau(z) \neq z$. 
\item[(ii)] $\mathcal{D}^{\overline{v}} g (F) = g \mathcal{D}^{\overline{v}} (F)$ if $F$ and $g(F)$ are in ${\mathcal F}_{ij} \otimes {\mathcal V}_{\rm gen}$.
\item[(iii)] $\mathcal{D}^{\overline{v}} ( (x-y)g (F)) = g( {\rm ev} (\overline{v}) F)$ if $F \in {\mathcal F}_{ij} \otimes {\mathcal V}_{\rm gen}$.
\end{itemize}
\end{lemma}
\begin{proof}
Since $\mathcal{D}^{\overline{v}}$ is linear, it is enough to show each of the statements for $g = E_{rs}$ and $F = f T(v+z)$ with  generic $v$ and $f \in {\mathcal F}_{ij}$.  

(i) Since $\tau(z) \neq z$, $e_{rs} (\sigma (v+z)) \in {\mathcal F}_{ij}$ for all $\sigma \in \Phi_{rs}$. Thus 
\begin{eqnarray*}
E_{rs} ( T(\overline{v}+z) )& = & \mathcal{D}^{\bar{v}}((x - y)E_{rs}(T(v + z))) \\ 
& = & \mathcal{D}^{\bar{v}} \left( \sum_{\sigma \in \Phi_{rs}} (x-y)e_{rs} (\sigma (v+z)) T(v + z+\sigma (\varepsilon_{rs}))  \right) \\
& = &  \sum_{\sigma \in \Phi_{rs}}  {\rm ev} (\bar{v}) \left ( e_{rs} (\sigma (v+z))  T(v + z+\sigma (\varepsilon_{rs}))  \right) \\
& = & \mbox{ev} ( \overline{v}) E_{rs} ( T(v+z)). 
\end{eqnarray*}

(ii) Using (i) and the facts that $E_{rs} ( T(v+z)) $ is in  ${\mathcal F}_{ij} \otimes {\mathcal V}$ and $E_{rs}(\mathcal{D}T(\bar{v} + w))=\mathcal{D}^{\overline{v}}( E_{rs} ( T(v + w)))$ we have  
$$
\mathcal{D}^{\overline{v}} E_{rs} ( f T(v+z)) - E_{rs} \mathcal{D}^{\overline{v}} ( f T(v+z)) = \mathcal{D}^{\overline{v}} (f) \left( \mbox{ev} (\overline{v}) E_{rs} ( T(v+z))  -  E_{rs} ( T(\overline{v}+z))  \right) =0.
$$

(iii) Taking into consideration that $ \mbox{ev}(\overline{v}) \left(  (x-y)E_{rs}(T(v + z))\right) = 0$, we have
\begin{eqnarray*}
 \mathcal{D}^{\overline{v}} \left( (x-y)E_{rs}(fT(v + z)) \right) & = & f(\overline{v}) \mathcal{D}^{\overline{v}} \left( (x-y)E_{rs}(T(v + z)) \right)\\ 
& = & E_{rs} ( \mbox{ev}(\overline{v}) f T(v+z)).
\end{eqnarray*}

\end{proof}

\begin{proposition} \label{t-v-rep}
Let $g_1,g_2 \in \mathfrak{gl} (n)$. Then
$$
[g_1,g_2] (T(\overline{v} + z)) = g_1(g_2 (T(\overline{v} + z))) - g_2(g_1 (T(\overline{v} + z))).
$$
\end{proposition}
\begin{proof}
We have that $g_t T(\overline{v} + z) =  \mathcal{D}^{\overline{v}} (G_t)$ where  $G_t = (x-y) g_t (T(v + z))$ is in ${\mathcal F}_{ij} \otimes V(T(v))$ for $t=1,2$. Hence 
\begin{eqnarray*}
 g_1(g_2 (T(\overline{v} + z))) - g_2(g_1 (T(\overline{v} + z))) &=&  g_1\mathcal{D}^{\overline{v}} (G_2) -  g_2 \mathcal{D}^{\overline{v}} (G_1)\\
 & = &  \mathcal{D}^{\overline{v}} (g_1 G_2) -  \mathcal{D}^{\overline{v}} (g_2 G_1)\\
& = &  \mathcal{D}^{\overline{v}} \left( (x-y) (g_1 g_2  - g_2 g_1) T(v+z) \right)\\
& = &  (g_1 g_2  - g_2 g_1) T(\overline{v}+z)
\end{eqnarray*}
 by  Lemma \ref{dij-commute}(i) and definitions.
\end{proof}

The following proposition is proved in \S \ref{proof-d-t}.
\begin{proposition} \label{d-t-v-rep}
Let $g_1,g_2 \in \mathfrak{gl} (n)$ and $\tau(z) \neq z$. Then
\begin{equation}\label{representation for derivative tableaux}
[g_1,g_2] (\mathcal{D}T(\bar{v} + z)) = g_1(g_2 (\mathcal{D}T(\bar{v} + z))) - g_2(g_1 (\mathcal{D}T(\bar{v} + z))).
\end{equation}
\end{proposition}

Combining Propositions \ref{t-v-rep} and \ref{d-t-v-rep} we obtain the following.

\begin{theorem}
If $\bar{v}$ is an $1$-singular vector in ${\mathbb C}^{\frac{n(n-1)}{2}}$, then  $ V(T(\bar{v}))$ is a $\mathfrak{gl} (n)$-module.
 \end{theorem}

\section{Action of the  Gelfand-Tsetlin subalgebra of $\mathfrak{gl}(n)$ on the module $V(T(\bar{v}))$} \label{sec-action}

Recall that $\gamma_{rs}(w)=\sum_{i=1}^r
(w_{ri}+r-1)^s \prod_{j\ne i} \left( 1 -
\frac{1}{w_{ri}-w_{rj}} \right)$ is a symmetric polynomial in $w_{r1}, \ldots, w_{rr}$ of degree $s$, and that $\{\gamma_{r1}, \ldots, \gamma_{rr}\}$ generate the algebra of symmetric polynomials in $w_{r1}, \ldots, w_{rr}$. 

Also, recall that  for a generic vector $v$ and $z\in\mathbb{Z}^{\frac{n(n-1)}{2}}$, we have
$$c_{rs}(T(v+z)) \  = \gamma_{rs}(v+z) T(v+z),$$
where $c_{rs}$ are the generators of $\Gamma$ defined in (\ref{equ_3}). Recall the fixed set of representatives $\mathcal S$ from Definition \ref{def-s}.

\begin{lemma}\label{generators of Gamma acting in many tableaux} Let $v$ be any nonzero generic vector in $\mathcal{S}$ such that $v_{ki}=x$ and $v_{kj}=y$.
\begin{itemize}
\item[(i)] If $z\in\mathbb{Z}^{\frac{n(n-1)}{2}}$ is such that $|z_{ki}-z_{kj}|\geq n-m$ for some $0 \leq m \leq n$, then for each $1\leq r \leq s\leq n-m$ we have:
\begin{itemize}
\item[(a)]
$
c_{rs}(T(\bar{v}+z))=\mathcal{D}^{\overline{v}}((x-y)c_{rs}(T(v+z))),
$
\item[(b)]
$c_{rs}(\mathcal{D}T(\bar{v}+z)))= \mathcal{D}^{\overline{v}}(c_{rs}(T(v+z)))$ if $z\neq \tau(z)$,
\end{itemize}
\item[(ii)] If $1\leq s\leq r\leq k$ and $z\in\mathbb{Z}^{\frac{n(n-1)}{2}}$ then the action of $c_{rs}$ on $T(\bar{v}+z)$ and 
$\mathcal{D}T(\bar{v}+z))$ is defined by the formulas in {\rm (i)}.

\end{itemize}
\end{lemma}

\begin{proof}
\begin{enumerate}[(i)]
\item By the hypothesis $|z_{ki}-z_{kj}|\geq n-m$, the coefficients  that appear in the decompositions of the vectors $(x-y)E_{i_s, i_1}(T(v+z))$, $E_{i_s, i_1}(T(v+z))$ ($z\neq\tau(z)$), $(x-y)E_{i_t, i_{t+1}}\ldots E_{i_s, i_1}(T(v+z))$ and $E_{i_t, i_{t+1}}\ldots E_{i_s, i_1}(T(v+z))$ ($z\neq\tau(z)$) for $1\leq t\leq s-1$ are all in $\mathcal{F}_{ij}$.
\begin{enumerate}[(a)]
\item For each $(i_{1},\ldots,i_{s})\in\{1,\ldots,r\}^{s}$ we have $(x-y)E_{i_s, i_1}(T(v+z))\in{\mathcal F}_{ij} \otimes {\mathcal V}_{\rm gen}$ and for each $1\leq t\leq s-1$, we have $(x-y)E_{i_t, i_{t+1}}\ldots E_{i_s, i_1}(T(v+z))\in{\mathcal F}_{ij} \otimes {\mathcal V}_{\rm gen}$. Then the statement follows from Lemma \ref{dij-commute}(ii). 

\item If $z\neq \tau(z)$ then for each $(i_{1},\ldots,i_{s})\in\{1,\ldots,r\}^{s}$, $E_{i_s i_1}(T(v+z))\in{\mathcal F}_{ij} \otimes {\mathcal V}_{\rm gen}$ and for each $1\leq t\leq s-1$,  $E_{i_t i_{t+1}}\ldots E_{i_s i_1}(T(v+z))\in{\mathcal F}_{ij} \otimes {\mathcal V}_{\rm gen}$. Hence we can use Lemma \ref{dij-commute}(ii).
\end{enumerate}
\item As $1\leq s\leq r\leq k$ then every tableau that appears in the decomposition of $(x-y)E_{i_s i_1}(T(v+z))$, $E_{i_s i_1}(T(v+z))$ ($z\neq\tau(z)$), $(x-y)E_{i_t i_{t+1}}\ldots E_{i_s i_1}(T(v+z))$ and $E_{i_t i_{t+1}}\ldots E_{i_s i_1}(T(v+z))$ ($z\neq\tau(z)$) for $1\leq t\leq s-1$ has the same  $(k,i)$th and $(k,j)$th entries. So, all of the listed vectors are in ${\mathcal F}_{ij} \otimes {\mathcal V}_{\rm gen}$ and using Lemma \ref{dij-commute}(ii) we complete the proof.\end{enumerate} \end{proof}

\begin{lemma}\label{lem-ck2}
If $\tau(z)\neq z$ then we have the following identities.
\begin{itemize}
\item[(i)] $c_{k2}(T(\bar{v}+z))=\gamma_{k2}(\bar{v}+z)T(\bar{v}+z)$ 
\item[(ii)] $(c_{k2}-\gamma_{k2}(\bar{v}+z))\mathcal{D}T(\bar{v}+z)\neq 0$.
\item[(iii)] $(c_{k2}-\gamma_{k2}(\bar{v}+z))^{2}\mathcal{D}T(\bar{v}+z)=0.$
\end{itemize}

\end{lemma}

\begin{proof}
If $w$ is a generic vector then $ c_{k2}T(w)=\gamma_{k2}(w)T(w)$ where $\gamma_{k2}(w)$ is a quadratic symmetric polynomial in variables $w_{k1},\ldots,w_{kk}$.
\begin{enumerate}[(i)]
\item  By Lemma \ref{generators of Gamma acting in many tableaux}(ii),  we have 
\begin{align*}
c_{k2}(T(\bar{v}+z))=&\mathcal{D}^{\overline{v}}((x-y)c_{k2}(T(v+z)))\\
=&\mathcal{D}^{\overline{v}}((x-y)\gamma_{k2}(v+z)T(v+z))\\
=&\gamma_{k2}(\bar{v}+z)T(\bar{v}+z)).
\end{align*} 
\item Also by Lemma \ref{generators of Gamma acting in many tableaux}(ii) we have:
\begin{align*}
c_{k2}(\mathcal{D}T(\bar{v}+z)))=& \mathcal{D}^{\overline{v}}(c_{k2}(T(v+z)))\\
=& \mathcal{D}^{\overline{v}}(\gamma_{k2}(v+z)T(v+z))\\
=& \mathcal{D}^{\overline{v}}(\gamma_{k2}(v+z))T(\bar{v}+z))+\gamma_{k2}(\bar{v}+z)\mathcal{D}T(\bar{v}+z))
\end{align*}
With $\mathcal{D}^{\overline{v}}(\gamma_{k2}(v+z))=2a(z_{ki}-z_{kj})\neq 0$ where $a$ is the coefficient of $(v_{ki}+z_{ki})^{2}$ in $\gamma_{k2}(v+z)$.
\item  This part follows from (i) and (ii).
\end{enumerate}
\end{proof}
\begin{lemma}\label{Gamma k separates tableaux}
Let $\Gamma_{k-1}$ be the subalgebra of $\Gamma$ generated by $\{c_{rs}:1\leq s\leq r\leq k-1\}$. If $z,z'\in\mathbb{Z}^{\frac{n(n-1)}{2}}$ are such that $z_{rs}\neq z'_{rs}$ for some $1\leq s\leq r\leq k-1$. Then, $\Gamma_{k-1}$ separates the tableaux $T(\bar{v}+z)$ and $T(\bar{v}+z')$, that is, there exis $c\in\Gamma_{k-1}$ and $\gamma\in\mathbb{C}$ such that $(c-\gamma)T(\bar{v}+z)=0$ but $(c-\gamma)T(\bar{v}+z')\neq 0$.
\end{lemma}

\begin{proof}
By Lemma \ref{generators of Gamma acting in many tableaux}(ii), we know that the action of the generators of $\Gamma_{k-1}$ on $T(\bar{v}+z)$ and $T(\bar{v}+z')$  is given by symmetric polynomials in the entries of rows $k-1,\ldots,1$. Assume the contrary, i.e. that $T(\bar{v}+z)$ and $T(\bar{v}+z')$ have the same character associated with the generators of $\Gamma_{k-1}$. Like in the generic case, the latter implies that one of the tableaux is obtained from the other by the action of an element in $S_{k-1}\times \cdots \times S_{1}$.  But the difference of the entries on the rows  $k-1,\ldots,1$ of $T(\bar{v}+z)$ and $T(\bar{v}+z')$ are integers. Hence, the tableaux must have different characters which leads to a contradiction. \end{proof}

For any $m\in\mathbb{Z}_{\geq 0}$ let $R_{m}$ be the set of $z\in\mathbb{Z}^{\frac{n(n-1)}{2}}$ such that $|z_{ki}-z_{kj}|=m$.

\begin{lemma}\label{lem-connect}
 If $z\in R_{m}$ then  there exists $\bar{z}\in R_{m+1}$ such that $T(\bar{v}+z)$ appears with non-zero coefficient in the decomposition of $E_{k+1,k-t}T(\bar{v}+\bar{z})$ for some $t\in\{0,1,\ldots,k-1\}$.
\end{lemma}

\begin{proof}
If $\bar{v}_{ki}=\bar{v}_{kj}=x$ and $w=\bar{v}+z$ then there exist $l\in\mathbb{Z}$ such that $w_{ki}=x+m-l$ and $w_{kj}=x-l$. Let $t$ be such that 
$$\begin{cases}
w_{ki}+1=w_{k-1,s_{k-1}} & \text{ for some \ \ \ $1\leq s_{k-1}\leq k-1$}\\
w_{k-1,s_{k-1}}+1=w_{k-2,s_{k-2}} & \text{ for some \ \ \ $1\leq s_{k-2}\leq k-2$}\\
\ \ \ \ \ \ \ \ \ \  \vdots & \ \ \ \ \ \ \ \ \ \  \vdots\\
w_{k-t+1,s_{k-t+1}}+1=w_{k-t,s_{k-t}} & \text{ for some \ \ \ $1\leq s_{k-t}\leq k-t$}\\
w_{k-t}+1\neq w_{k-t-1,s} & \text{ for any \ \ \ $1\leq s\leq k-t-1$}.\\
\end{cases}
$$
Then the coefficient of $T(\bar{v}+z)$ in the decomposition of $E_{k+1,k-t}T(\bar{v}+\bar{z})$  is not zero, where $\bar{z}=z+\delta^{ki}+\delta^{k-1, s_{k-1}}+\ldots+\delta^{k-t, s_{k-t}}\in R_{m+1}$.
\end{proof}
\begin{remark}\label{elements in Rm}
All  tableaux that appear in the decomposition of $E_{k+1,k-t}T(\bar{v}+\bar{z})$ are of the form $T(\bar{v}+\bar{z}+\sigma(\varepsilon_{k+1,k-t}))$ for  $\sigma\in\Phi_{k+1,k-t}$. Furthermore, we have:
$$
\begin{cases}
\bar{z}+\sigma(\varepsilon_{k+1,k-t})\in R_{m+2} & \text{ if } \ \ \ \  \sigma[k]=(1,j)\\
\bar{z}+\sigma(\varepsilon_{k+1,k-t})\in R_{m} & \text{ if } \ \ \ \  \sigma[k]=(1,i)\\
\bar{z}+\sigma(\varepsilon_{k+1,k-t})\in R_{m+1} & \text{ if } \ \ \ \  \sigma[k]\notin \{(1,i),(1,j)\}\\
\end{cases}
$$
In particular, all  tableaux $T(\bar{v}+\bar{z}+\sigma(\varepsilon_{k+1,k-t}))$ with $\bar{z}+\sigma(\varepsilon_{k+1,k-t})\in R_{m}$ have the same entries in rows $k,\ldots,n$ and two of these tableaux have at least one different entry in rows $1,\ldots,k-1$. Therefore, by Lemma \ref{Gamma k separates tableaux} all the tableaux $T(\bar{v}+\bar{z}+\sigma(\varepsilon_{k+1,k-t}))$ with $\bar{z}+\sigma(\varepsilon_{k+1,k-t})\in R_{m}$ have different Gelfand-Tsetlin  characters. 
 \end{remark}

\begin{theorem}\label{GT module structure}
The module $V(T(\bar{v}))$ is an $1$-singular  Gelfand-Tsetlin module. Moreover for any $z\in\mathbb{Z}^{\frac{n(n-1)}{2}}$ and any $1\leq r \leq s\leq n$ the following identities hold.
\begin{itemize}
\item[(i)]
$
c_{rs}(T(\bar{v}+z))=\mathcal{D}^{\overline{v}}((x-y)c_{rs}(T(v+z)))
$
\item[(ii)]
$c_{rs}(\mathcal{D}T(\bar{v}+z)))= \mathcal{D}^{\overline{v}}(c_{rs}(T(v+z)))$ if $z\neq \tau(z)$.
\end{itemize}
\end{theorem}
\begin{proof}
Let $R_{\geq n} := \cup_{m\geq n} R_m$. For any $z\in R_{\geq n}$ consider the submodule $W_{z}$ of $V(T(\bar{v}))$ generated by $T(\bar{v}+z)$. By Lemma \ref{generators of Gamma acting in many tableaux}(i)(a), $T(\bar{v}+z)$ is a common eigenvector of all generators of $\Gamma$ and thus $W_{z}$ is a Gelfand-Tsetlin module by Lemma \ref{lem-cyclic-Gelfand-Tsetlin}. Then $W=\sum_{z\in R_{\geq n}}W_{z}$  is also a Gelfand-Tsetlin module. We first show that $W$ contains all tableau $T(\bar{v}+z)$ for any $z\in\mathbb{Z}^{\frac{n(n-1)}{2}}$.
   Indeed, assume that $|z_{ki}-z_{kj}|= n-1$ and consider $T(\bar{v}+z)$. Then, by Lemma \ref{lem-connect} there exists $z'\in R_n$ and a nonzero $x\in \gl_{k+1}$ such that  $xT(\bar{v}+z')=\sum_{t=0}^N a_t T(\bar{v}+z^{(t)})$, where $a_t\in \mathbb C$, $z^{(0)}=z$ and, $|z_{ki}^t-z_{kj}^{(t)}|\geq n-1$.

Following Remark \ref{elements in Rm}, we assume that $N=1$ and $z^{(0)} , z^{(1)}\in R_{n-1}$ without loss of generality since $z^{(m)}$ in $R_{\geq n}$ implies $T(\bar{v}+z^{(m)})\in W$. The action of  all generators $\{c_{rs}\}_{1\leq r \leq s\leq n-1}$ of $\Gamma$, except for the center of $U$, on  $T(\bar{v}+z)$ and $T(\bar{v}+z^{(1)})$ is determined by Lemma \ref{generators of Gamma acting in many tableaux}. Let $c\in \Gamma$ be a central element and $(c-\gamma)T(\bar{v}+z')=0$ for some complex $\gamma$. Then $(c-\gamma)xT(\bar{v}+z')=0$ $=(c-\gamma) (a_0 T(\bar{v}+z)+ a_1 T(\bar{v}+z^{(1)}))$.  Recall that by Lemma \ref{Gamma k separates tableaux}, there exists $C\in \Gamma_{k-1}$ which  acts with different scalars on $T(\bar{v}+z)$ and $T(\bar{v}+z^{(1)})$. Since $C$ commutes with $(c-\gamma)$, both $T(\bar{v}+z)$ and $T(\bar{v}+z^{(1)})$ are in $W$. Moreover,
  $(c-\gamma)T(\bar{v}+z)=$ $(c-\gamma)T(\bar{v}+z^{(1)})=0$. Hence, the action of $\Gamma$ on any $T(\bar{v}+z)$ with $z\in R_{n-1}$ is as in Lemma \ref{generators of Gamma acting in many tableaux}. Moreover, $T(\bar{v}+z)\in W$ for any $z\in R_{n-1}$. 
  Next we consider a tableau $T(\bar{v}+z)$ with $z\in R_{n-2}$. Again 
  by Lemma~\ref{lem-connect} one finds a nonzero $y\in \gl_{k+1}$ and
  $z'\in R_{n-1}$ such that  $y T(\bar{v}+z')$ contains $T(\bar{v}+z)$ and at most one other tableau. For all generators of centers of $U(\gl_m))$, $m\leq n-2$ the statement follows from Lemma \ref{generators of Gamma acting in many tableaux}. 
  If $c$ is in the center of $U$ or in the center of $U(\gl_{n-1})$ then it commutes with $y$.  Choose $C\in \Gamma_{k-1}$ which separates the tableaux in the image $y T(\bar{v}+z')$ and which acts by a scalar on the tableau 
  $T(\bar{v}+z')$. Applying the argument above we conclude that the action of $\Gamma$ on any $T(\bar{v}+z)$ with $z\in R_{n-2}$ is determined  by Lemma \ref{generators of Gamma acting in many tableaux} and $T(\bar{v}+z)\in W$ for any $z\in R_{n-2}$. 
 Continuing analogously with the sets 
  $R_{n-3}, \ldots, R_0$ we show that any tableau $T(\bar{v}+z)$  belongs to $W$. Note that when $z\in R_0$, $\tau(z)=z$. In this case it will be the unique such tableau coming from some  $T(\bar{v}+z')$ with $z'\in R_1$ and ``separation'' is not needed.   
  
  Consider the  quotient $\overline{W}=V(T(\bar{v}))/W$.
 The vector $\mathcal{D}T(\bar{v}+z))+W$ of $\overline{W}$ is a common eigenvector of $\Gamma$ by Lemma \ref{generators of Gamma acting in many tableaux}(i)(b) for any $z\in R_n$. We can repeat now the argument above substituting everywhere the tableaux 
  $T(\bar{v}+z)$ by $\mathcal{D}T(\bar{v}+z))$. Hence, $\overline{W}=\sum_{z\in R_n}\overline{W}_{z}$, where 
  $\overline{W}_{z}$ denotes the submodule of $\overline{W}$ generated by  $\mathcal{D}T(\bar{v}+z)) +W$. 
  By Lemma \ref{lem-cyclic-Gelfand-Tsetlin} we conclude that $\overline{W}$ is a Gelfand-Tsetlin module. Therefore, 
  $V(T(\bar{v}))$ is a Gelfand-Tsetlin module with the required action of the generators of $\Gamma$.  
\end{proof}

The above theorem implies the following.
\begin{corollary}\label{ck2 acts as jordan cell} For any $z, w\in\mathbb{Z}^{\frac{n(n-1)}{2}}$ with $w\neq\tau(w)$, the action of the generators of $\Gamma$ on $V(T(\bar{v}))$ is given by: 
\begin{align*}
c_{rs}(T(\bar{v}+z)&=\gamma_{rs}(\bar{v}+z)T(\bar{v}+z)\\
c_{rs}(\mathcal{D}T(\bar{v}+w)))&=\gamma_{rs}(\bar{v}+w)\mathcal{D}T(\bar{v}+w))+\mathcal{D}^{\overline{v}}(\gamma_{rs}(v+w))T(\bar{v}+w).
\end{align*}
Moreover, $\dim V(T(\bar{v}))(\sm)\leq 2$ for any $\sm\in\Sp\Ga$. 
\end{corollary}

\subsection{Example} Let $n=3$, $a$ any complex number, $\bar{v}=(a,a,a|a,a|a)$ and $z=(m,n,k)\in\mathbb{Z}^{3}$. In this case the module $V(T(\bar{v}))$ has length $10$. 
To describe the irreducible subquotients  of $V(T(\bar{v}))$ we use the notation
$$
L(D) = \Span \{ P(z):z\in D\}
$$
for a set of inequalities $D$, where 

\begin{center}
$P(z):=\begin{cases}
 T(\bar{v}+z)&
 \text{ if } m\leq n\\
 \mathcal{D}T(\bar{v}+z)&
 \text{ if } m> n\\
 \end{cases}$
\end{center}

The  irreducible subquotients  of $V(T(\bar{v}))$ can be separated into two groups:

\begin{enumerate}[(i)]

\item Eight irreducible Verma modules. 

{\scriptsize $$L_{1}=L\left( \begin{cases}
m\leq n\\
n\leq 0\\
k\leq n
\end{cases}\bigcup
\begin{cases}
m> n\\
m\leq 0\\
k\leq n
\end{cases}\right); L_{3}=L\left( \begin{cases}
m\leq n\\
n\leq 0\\
k>n
\end{cases}\bigcup
\begin{cases}
m> n\\
m\leq 0\\
k>n
\end{cases}\right);$$

$$L_{5}=L\left(\begin{split}
m\leq n\\
m\leq 0\\
n>0\\
k\leq m
\end{split}\right)\cong L\left(
\begin{split}
m> n\\
n\leq 0\\
m>0\\
k\leq n
\end{split}\right)=L'_{5}$$}
{\scriptsize $$L_{2}=L\left( \begin{cases}
m\leq n\\
m> 0\\
k>n
\end{cases}\bigcup
\begin{cases}
m> n\\
n> 0\\
k>n
\end{cases}\right);\text{   } L_{4}=L\left( \begin{cases}
m\leq n\\
m> 0\\
k\leq n
\end{cases}\bigcup
\begin{cases}
m> n\\
n>0\\
k\leq n
\end{cases}\right);$$

$$L_{6}=L\left(\begin{split}
m\leq n\\
m\leq 0\\
n>0\\
k> n
\end{split}\right)\cong L\left(
\begin{split}
m> n\\
n\leq 0\\
m>0\\
k>m
\end{split}\right)=L'_{6}$$}

\item Two isomorphic modules with infinite dimensional weight spaces:

{\scriptsize $$L_{7}=L\left(\begin{split}
m\leq n\ \ \ \\
m\leq 0\ \ \ \\
n>0\ \ \ \\
m<k\leq n
\end{split}
\right)\cong L\left(\begin{split}
m> n\ \ \ \\
n\leq 0\ \ \ \\
m>0\ \ \ \\
n<k\leq m
\end{split}
\right)=L'_{7}$$}
\end{enumerate}

The Loewy decomposition of $V(T(\bar{v}))$ has successive components:
$$L_{1} ,\  L_{3}\oplus L_{5} ,\  L_{7} ,\  L'_{5}\oplus L_{6} ,\  L'_{7} ,\  L_{4}\oplus L'_{6} ,\  L_{2}$$
where $L_1$ is  the socle of  $V(T(\bar{v}))$. 

\begin{remark}An interesting observation is that the module $L_7$ (and hence $L_7'$) can be described in terms of the localization functor $D_{21}$ with respect to the Ore subset $\{ E_{21}^k\; | \; k \in {\mathbb Z}_{\geq 0}\} $ of $U$ (we refer the reader to  \cite{Deo} for a definition and properties of the functor $D_{21}$). More precisely we have the following exact sequence
$$
 0 \to L_7 \to (D_{21}L_5)/L_5 \to L_6 \to 0.
$$
In fact, as we prove in  \cite{FGR}, every irreducible Gelfand-Tsetlin $\mathfrak{gl} (3)$-module can be obtained as a subquotient of a localized Verma module.
\end{remark}


\section{Irreducibility of $V(T(\bar{v}))$}

We introduce the following notation which will be used in this section only.
\begin{eqnarray*}
L &:=& \{T(\bar{v}+z): z\neq\tau(z)\},\\
S& :=& \{T(\bar{v}+z): z=\tau(z)\},\\
D& := &\{\mathcal{D}T(\bar{v}+z): z\neq\tau(z)\}.
\end{eqnarray*}

\begin{lemma}\label{the module is generated by two elements}
Let $z\in\mathbb{Z}^{\frac{n(n-1)}{2}}$ such that $z\neq\tau(z)$ and $w=\bar{v}+z$. If $w_{rs}-w_{r-1,t}\notin \mathbb{Z}_{\geq 0}$ for any $r,s,t$, then the module $V(T(\bar{v}))$ is generated by the two tableau $T(\bar{v}+z)$ and $\mathcal{D}T(\bar{v}+z))$.
\end{lemma}

\begin{proof}
The action of $\gl(n)$ on the elements from $L$  is given by  the classical Gelfand-Tsetlin formulas (see Proposition \ref{coefficients e_ij}), and the conditions $w_{rs}-w_{r-1,t}\notin \mathbb{Z}_{\geq 0}$ imply that $\Omega^{+}(T(\bar{v}+z))=\emptyset$ (see  Definition \ref{definition of Omega +}). Hence, by Theorem \ref{Basis for irreducible generic modules gl(n)}(i), the submodule  generated  by $T(\bar{v}+z)$ contains  $L\bigcup S$. Now, given any $w\neq\tau(w)$ and $\mathcal{D}T(\bar{v}+w) \in D$,  we have $E_{rs}(\mathcal{D}T(\bar{v}+w))=\sum_{\sigma\in\Phi_{rs}}\mathcal{D}^{\overline{v}} (e_{rs}(\sigma(v+w))T(\bar{v}+w+\sigma(\varepsilon_{rs}))+\sum_{\sigma\in\Phi_{rs}}(e_{rs}(\sigma(\bar{v}+w))\mathcal{D}T(\bar{v}+w+\sigma(\varepsilon_{rs}))$. The condition  $w_{rs}-w_{r-1,t}\notin \mathbb{Z}_{\geq 0}$ for any $r,s,t$ implies again that $\Omega^{+}(\mathcal{D}T(\bar{v}+z))=\emptyset$. Thus, by Theorem \ref{Basis for irreducible generic modules gl(n)}, the submodule generated by the tableau $\mathcal{D}T(\bar{v}+z)$  contains  $D$. 
\end{proof}

\begin{lemma}\label{when a derivative tableau generates}
Let $z\in\mathbb{Z}^{\frac{n(n-1)}{2}}$ such that $z\neq\tau(z)$ and $w=\bar{v}+z$. If $w_{rs}-w_{r-1,t}\notin \mathbb{Z}_{\geq -1}$ for any $r,s,t$, then $V(T(\bar{v}))$ is generated by $\mathcal{D}T(\bar{v}+z))$.
\end{lemma}
\begin{proof}

By Lemma \ref{the module is generated by two elements} the module $V(T(\bar{v}))$ is generated by $\mathcal{D}T(\bar{v}+z))$ and any tableau of the form $T(\bar{w})$ such that $\bar{w}_{rs}-\bar{w}_{r-1,t}\notin \mathbb{Z}_{\geq 0}$ for any $r,s,t$. Thus, it is enough to prove that we can generate one such $T(\bar{w})$ from $\mathcal{D}T(\bar{v}+z))$. To prove  this recall $E_{k-1,k}(\mathcal{D}T(\bar{v}+z))=\sum_{\sigma\in\Phi_{k-1,k}}\mathcal{D}^{\overline{v}}(e_{k-1,k}(\sigma(v+z))T(\bar{v}+z+\sigma(\varepsilon_{k-1,k}))+\sum_{\sigma\in\Phi_{k-1,k}}{\rm ev}(\bar{v})(e_{k-1,k}(\sigma(v+z)))\mathcal{D}T(\bar{v}+z+\sigma(\varepsilon_{k-1,k}))$. But since
$$\mathcal{D}^{\overline{v}}(e_{k-1,k}(v+z))=-\frac{1}{2}\left(\frac{(z_{kj}-z_{ki})\prod_{t\neq i,j}^{k}((\bar{v}+z)_{k-1,1}-(\bar{v}+z)_{k,t})}{\prod_{t\neq 1}^{k-1}((\bar{v}+z)_{k-1,1}-(\bar{v}+z)_{k-1,t})}\right)\neq 0,
$$
the tableau $T(\bar{v}+z+\varepsilon_{k-1,k})$ satisfies the desired conditions.
\end{proof}

\begin{lemma}\label{when a non derivative tableau generates}
If $z\neq\tau(z)$ and $\bar{v}_{rs}-\bar{v}_{r-1,t}\notin \mathbb{Z}$ for any $1\leq t <r\leq n$, $1\leq s\leq r$  then $T(\bar{v}+z)$ generates $V(T(\bar{v}))$.
\end{lemma}
\begin{proof}

As $V(T(\bar{v}))$ is a Gelfand-Tsetlin module, the elements of the basis of $V(T(\bar{v}))$ can be separated by characters of $\Gamma$ and by Lemma \ref{lem-ck2} we can separate different tableaux with the same Gelfand-Tsetlin character by the action of $c_{k2}$. And since a submodule of a Gelfand-Tsetlin module is a Gelfand-Tsetlin module, to prove the lemma it is sufficient to show that any tableau of $V(T(\bar{v}))$ appears with a nonzero coefficient in the decomposition of $gT(\bar{v}+z)$ for some $g\in U$.

We prove the lemma in three steps: first  from  $T(\bar{v}+z)$ we  generate the set $L\bigcup S$, then from $S$ we  generate one tableau in $D$, and last, from this one tableau in $D$  we  generate the whole set $D$. 

{\it Step 1: All tableaux of $L \cup S$ are in $U(T(\bar{v}+z)$).} The set $L\bigcup S$ can be obtain in this way because the action of $\mathfrak{gl}(n)$ on $L$ is given by the classical Gelfand-Tsetlin formulas, so the coefficients in the formulas have  numerators which are  products of nonzero monomials of the form $\bar{v}_{rs}+w_{rs}-\bar{v}_{r-1,t}-w_{r-1,t}$. 

{\it Step 2: One tableau of $D$ is in $U(T(\bar{v}+z))$.} If $z=\tau(z)$ then, the coefficient of $\mathcal{D}T(\bar{v}+z+\sigma(\varepsilon_{k+1,k}))$ on $E_{k+1,k}(T(\bar{v}+z))$ is ${\rm ev}(\bar{v})((x-y)e_{k+1,k}(\sigma(v+z)))$, so in order to obtain an element of $D$ from $S$ we have to find some $\sigma$ such that ${\rm ev}(\bar{v})((x-y)e_{k+1,k}(\sigma(v+z)))\neq 0$.  Take $\sigma$ with $\sigma[k]=(1,i)$ and all other $\sigma[t] = {\rm Id}$. Then  the numerator of ${\rm ev}(\bar{v})((x-y)e_{k+1,k}(\sigma(v+z)))$ is $$\prod_{j=1}^{k-1}(\bar{v}_{ki}+z_{ki}-\bar{v}_{k-1,j}-z_{k-1,j})\neq 0.$$
{\it Step 3: All tableaux of $D$ are in $U(T(\bar{v}+z))$.} Given any $w\neq\tau(w)$ and $\mathcal{D}T(\bar{v}+w) \in D$, the coefficient of $\mathcal{D}T(\bar{v}+w+\sigma(\varepsilon_{rs})))$ on $E_{rs}(\mathcal{D}T(\bar{v}+w))$ is $e_{rs}(\sigma(\bar{v}+w))$. Now, using the lemma  hypothesis, we see that $e_{rs}(\sigma(\bar{v}+w))\neq 0$ for any $\sigma$ implying that $\mathcal{D}T(\bar{v}+w)$ generates $D$.
\end{proof}
Lemmas \ref{when a non derivative tableau generates} and \ref{when a derivative tableau generates} together with  Lemma \ref{lem-ck2}  imply the following.

\begin{theorem}\label{thm-when L irred}
The module $V(T(\bar{v}))$ is irreducible whenever $\bar{v}_{rs}-\bar{v}_{r-1,t}\notin \mathbb{Z}$ for any $1\leq t <r\leq n$, $1\leq s\leq r$. 
\end{theorem}

\begin{remark}
We conjecture that the condition  $\bar{v}_{rs}-\bar{v}_{r-1,t}\notin \mathbb{Z}$ is also necessary for the irreducibility of $V(T(\bar{v}))$. The conjecture is true for $n=3$ as proved in \cite{FGR}.
\end{remark}

\section{Number of non-isomorphic irreducible modules associated with a singular character} \label{sec-last}

Let $\sm\in \Sp\Ga$ and $\ell_{\sm}\in \Sp \Lambda$ be such that $\pi(\ell_{\sm})=\sm$. Then $\ell_{\sm}$ defines a tableau $v=v_{\sm}$. Assume that $v_{\sm}$ is $1$-singular, that is $v_{ki}-v_{kj}\in \mathbb Z$ for the fixed $k,i,j$, and all other differences of $v_{tr}$ are non integer.  

Consider the module $M=U/U \sm$. Since $U$ is free as a left and as a right $\Ga$-module (\cite{Ovs}) then, $M\neq 0$. Also $M$ is a Gelfand-Tsetlin module by Lemma \ref{lem-cyclic-Gelfand-Tsetlin}. 
Let $\sn\in \Sp\Ga$. Recall that  by Theorem \ref{thm-bound-dim} 
we
 have  $$\dim M(\sn) \leq |\{\phi\in \mathcal M\mid \pi(\phi\ell_{\sm})=\sn\}|.$$
It is easy to see that the right hand side of the inequality above equals $2$. Hence we proved the following.

\begin{corollary}\label{cor-multiplicity}
All Gelfand-Tsetlin multiplicities of $M= U/U \sm$ are at most $2$. Moreover, if $w=w_{\sm}$ and
$w_{ki}=w_{kj}$, then $\dim M(\sm)=1$.

\end{corollary}

Now we are ready to present irreducible $1$-singular Gelfand-Tsetlin modules as subquotients of universal $1$-singular modules.

\begin{theorem}\label{thm-irr} Let $\sn\in \Sp\Ga$ be such that  $w=\ell_{\sn}$ is $1$-singular and $w=\bar{v}+z$.
\begin{itemize}
\item[(i)]
 There exist at most two non isomorphic irreducible Gelfand-Tsetlin modules $N_1$ and $N_2$ such that $N_1(\sn)\neq 0$ and  $N_2(\sn)\neq 0$.  
\item[(ii)]
 If 
$w_{ki}=w_{kj}$ then there exists a unique irreducible Gelfand-Tsetlin module $N$ with $N(\sn)\neq 0$. This module appears as a subquotient of   $V(T(\bar{v}))$.
\item[(iii)]If $w_{rs}-w_{r-1,t}\notin \mathbb{Z}$ for any $1\leq t <r\leq n$, $1\leq s\leq r$ then $V(T(\bar{v}))$ is the unique irreducible Gelfand-Tsetlin module $N$ with such that $N(\sn)\neq 0$. 

 \end{itemize}
\end{theorem}

\begin{proof}
 Let $X_{\sn}=U/U \sn$. We know that  $X_{\sn}=U/U \sn$ is a Gelfand-Tsetlin module by Lemma \ref{lem-cyclic-Gelfand-Tsetlin}. Furthermore, any irreducible Gelfand-Tsetlin module $M$ with $M(\sn)\neq 0$ is a homomorphic image of $X_{\sn}$, and  $X_{\sn}(\sn)$  maps onto $M(\sn)$. Since both spaces  $X_{\sn}(\sn)$ and $M(\sn)$ have additional structure as modules over certain algebra (see Corollary 5.3, \cite{FO2}) then the projection $X_{\sn}(\sn) \to M(\sn)$ is in fact a homomorphism of modules. Taking into account that $\dim X_{\sn}(\sn) \leq 2$ by Corollary \ref{cor-multiplicity}, we conclude that
there exist at most two non isomorphic irreducible  $N$ with $N(\sn)\neq 0$. This proves part (i) of the theorem. 

 Recall that  the Gelfand-Tsetlin multiplicities of $V(T(\bar{v}))$ are bounded by $2$ by Corollary \ref{ck2 acts as jordan cell}.   Assume now that $w_{ki}=w_{kj}$.  Then
$\dim V(T(w))(\sn)=1$. But this is possible  if and only if $\dim X_{\sn}(\sn)=1$ by Theorem \ref{thm-bound-dim}. Hence, there exists a unique 
irreducible quotient   $N$ of $X_{\sn}$ (and of $V(T(\bar{v}))$) with $N(\sn)\neq 0$ which implies part  (ii).  Finally, part (iii) follows from Theorem \ref{thm-when L irred}.
\end{proof}

\begin{remark} In order to complete the  classification of irreducible $1$-singular Gelfand-Tsetlin modules we need to find a presentation for two non-isomorphic irreducible modules $N_1$ and $N_2$ be  with $N_i(\sn)\neq 0$, $i=1,2$, and some $1$-singular $\sn$. We conjecture that both $N_1$ and $N_2$  appear as  subquotients of $V(T(\bar{v}))$ for $w_{\sn}= \bar{v}+z$.  By Theorem \ref{thm-irr}(i) this is  true if $V(T(\bar{v}))$ has two non-isomorphic irreducible  subquotients with $\sn$ in their support. On the other hand, if $V(T(\bar{v}))$ has two isomorphic irreducible subquotients, the conjecture remains open. Such irreducible modules have all Gelfand-Tsetlin multiplicities $1$ (the case considered in \cite{LP1}, \cite{LP2}). 
The existence of isomorphic subquotients of $V(T(\bar{v}))$ is present already in the case $n=3$, a case in which  the conjecture holds, as shown in \cite{FGR}.  
\end{remark}

\section{Appendix}

Like in Section \ref{sec-der}, in this appendix we assume that  all tableaux $T(v)$ have fixed first row, hence the corresponding vectors $v$ are in $\mathbb{C}^{\frac{n(n-1)}{2}}$. The goal of this appendix is to prove Propositions   \ref{compatible} and \ref{d-t-v-rep}.

\subsection{Useful identities}
For a function $f = f(v)$  by $f^{\tau}$ we denote the function $f^{\tau} (v) = f (\tau (v))$.

The following lemma can be easily verified. 
\begin{lemma}\label{dif operaator in functions}
Suppose $f\in {\mathcal F}_{ij}$ and $h:=\frac{f-f^{\tau}}{x-y}$.
\begin{itemize}
\item[(i)] If $f=f^{\tau}$, then ${\mathcal D}^{\bar{v}}(f)=0$.
\item[(ii)] If $h\in {\mathcal F}_{ij}$, then ${\rm ev} (\bar{v})(h)=2{\mathcal D}^{\bar{v}} (f)$.
\item[(iii)] ${\rm ev} (\bar{v})(f)={\mathcal D}^{\bar{v}}((x-y)f)$.
\end{itemize}
\end{lemma}

\begin{lemma} \label{dv-formulas}
Let $f_m, g_m$, $m=1, \ldots, t$, be functions such that $f_m,(x-y)g_m$, and $\sum_{m=1}^t f_m g_m$ are in ${\mathcal F}_{ij}$ and $g_m \notin  {\mathcal F}_{ij}$. Assume also that  $\sum_{m=1}^t f_m g_m^{\tau}= 0$. Then the following identities hold.
\begin{itemize}
\item[(i)] $2\sum_{m=1}^t{\mathcal D}^{\bar{v}} (f_m) {\mathcal D}^{\bar{v}} ( (x-y) g_m) =  {\mathcal D}^{\bar{v}} \left(\sum_{m=1}^t f_m g_m\right)$.
\item[(ii)]  $2\sum_{m=1}^t{\mathcal D}^{\bar{v}} (f_m) {\rm ev} (\bar{v}) ( (x-y) g_m) =  {\rm ev} (\bar{v})  \left(\sum_{m=1}^t f_mg_m\right)$.
\end{itemize}
\end{lemma}

\begin{proof} Set for simplicity $\bar{g}_m = (x-y)g_m$. For (i) we use Lemma \ref{dif operaator in functions} and obtain
\begin{eqnarray*}
{\mathcal D}^{\bar{v}} \left( \sum_{k=1}^t f_mg_m\right) & = & {\mathcal D}^{\bar{v}} \left(\sum_{m=1}^t f_mg_m+ \sum_{m=1}^t f_mg_m^{\tau}\right)\\
& = & {\mathcal D}^{\bar{v}} \left(\sum_{m=1}^t f_m \frac{\bar{g}_m - (\bar{g}_m)^{\tau}}{x-y}
\right)\\
& = &\sum_{m=1}^t   {\mathcal D}^{\bar{v}} (f_m) {\rm ev}(\bar{v}) \left( \frac{\overline{g}_m - (\overline{g}_m)^{\tau}}{x-y} \right) + \sum_{m=1}^t   {\rm ev}(\bar{v}) (f_m)  {\mathcal D}^{\bar{v}} \left( \frac{\overline{g}_m - (\overline{g}_m)^{\tau}}{x-y} \right)\\
& = &2 \sum_{m=1}^t  {\mathcal D}^{\bar{v}} (f_m) {\mathcal D}^{\bar{v}} (\overline{g}_m).
\end{eqnarray*}

For (ii) we use similar arguments.
\end{proof}

Given $z \in {\mathbb Z}^{\frac{n(n-1)}{2}}$, we define the following set:
\begin{eqnarray*}
\overline{\Phi}_{rs}&=& \{ \sigma \in \Phi_{rs} \: | \; \tau(z + \sigma(\varepsilon_{rs})) =  z + \sigma(\varepsilon_{rs}) \}.
\end{eqnarray*}
Also for any $(\sigma_1',\sigma_2') \in \bar{\Phi}_{r s}\times\bar{\Phi}_{\ell m}$ define
\begin{eqnarray*}
\Phi_{(\sigma_1',\sigma_2')}&=& \{ (\sigma_1,\sigma_2) \in \Phi_{r s}\times\Phi_{\ell m} \: | \; \sigma_1(\varepsilon_{rs})+\sigma_2(\varepsilon_{\ell m})=\sigma_1'(\varepsilon_{rs})+\sigma_2'(\varepsilon_{\ell m}) \}. 
\end{eqnarray*}

\begin{remark} We have that $(\sigma_{1},\sigma_{2})\in\bar{\Phi}_{rs}\times\bar{\Phi}_{\ell m}$. Indeed, if $(\sigma_1',\sigma_2') \in \bar{\Phi}_{rs}\times\bar{\Phi}_{\ell m}$ then $\Phi_{(\sigma_1',\sigma_2')}\subseteq \bar{\Phi}_{rs}\times\bar{\Phi}_{\ell m}$. But, if $(\sigma_{1},\sigma_{2})\in \Phi_{(\sigma_1',\sigma_2')}$ then $\{\sigma_{1}[k],\sigma_{2}[k]\}=\{\sigma_{1}'[k],\sigma_{2}'[k]\}$ and the elements of $\bar{\Phi}_{r s}$ or $\bar{\Phi}_{\ell m}$ satisfy conditions that involve  row $k$ only. 
\end{remark}
\begin{lemma} \label{lm-rs-ident} Let $v$ be generic, $\ell \neq m$, $r \neq s$ and let $z \in {\mathbb Z}^{\frac{n(n-1)}{2}}$ be such that  $\sigma_1' \in \overline{\Phi}_{rs} $, $\sigma_2' \in \overline{\Phi}_{\ell m }$. Then
$$
\sum \big(e_{\ell m} (\sigma_2 (v+z)) e_{rs} (\sigma_1 (v+z+ \sigma_2 (\varepsilon_{\ell m}))) - e_{rs } (\sigma_1 (v+z)) e_{\ell m} (\sigma_2 (v+z+ \sigma_1 (\varepsilon_{rs})))\big) =0.
$$
where the sum above is over $(\sigma_1, \sigma_2) \in \Phi_{(\sigma_1',\sigma_2')}$.
\end{lemma}
\begin{proof} The identity in the lemma follows by comparing the coefficient of $T(v+z + \sigma_1' (\varepsilon_{rs}) +   \sigma_2' (\varepsilon_{\ell m}) )$ on both sides of
\begin{equation} \label{eq-comp}
E_{rs}(E_{\ell m} T(v+z)) - E_{\ell m }(E_{r s} T(v+z))   = [E_{rs}, E_{\ell m}] T(v+z)
\end{equation}
 Note that if $w = \sigma_1' (\varepsilon_{rs}) +   \sigma_2' (\varepsilon_{\ell m})$ then 
 $$
 (w_{ki}, w_{kj})\in \{ (1,-1),(-1,1),(2,0),(-2,0), (0,2),(0,-2)\}
 $$ and in each of these six cases we have that the coefficient of $T(v+z + \sigma_1' (\varepsilon_{rs}) +   \sigma_2' (\varepsilon_{\ell m}) )$  in the right hand side of (\ref{eq-comp}) is zero. 
 \end{proof}

For $\min (\ell,m) \leq k \leq \max(\ell,m) -1$ and $1 \leq t \leq k$ we set
$$
\Phi_{\ell m} (k,t) = \{ \sigma \in \Phi_{\ell m} \; | \; \sigma[k] = (1,t)\}.
$$
In most of the considerations in this section we will need  $\Phi_{\ell m} (k,t)$ for $t=i$ and $t =j$ only. We set for convenience $\Phi_{\ell m} (i) = \Phi_{\ell m} (k,i)$ and $\Phi_{\ell m} (j)  = \Phi_{\ell m} (k,j)$.  The following lemma will be useful to prove that $V(T(\bar{v}))$ is a $\mathfrak{gl} (n)$-module.

\begin{lemma}\label{poles of order at most 1}
Let $z \in {\mathcal H}$ (equivalently, $\tau (z) = z$), $\sigma \in \Phi_{rs}$ and $w=v+z$. The function $e_{rs} (\sigma (w))$ of $v$ has a simple pole on $\overline{\mathcal H}$ if $\min (r, s ) \leq k \leq \max (r, s) - 1$ and $\sigma \in \Phi_{rs} (i)  \cup \Phi_{rs} (j)$. In all other cases $e_{rs} (\sigma (w))$ is in ${\mathcal F}_{ij}$, i.e. it is smooth on $\overline{\mathcal H}$.
\end{lemma}
\begin{proof}
The case $r=s$ is trivial, since $e_{rr}(\sigma(w))\in\mathcal{F}_{ij}$ for any $\sigma$. Suppose now that $r < s$. Then the denominator of $e_{rs}(\sigma(w))$ is 
$$\prod_{t=r}^{s-1}\left(\prod_{j\neq 1}^{t}(w_{t,\sigma[t](1)}-w_{t,\sigma[t](j)})\right),$$
which implies  the lemma. The case $r>s$ is analogous.
\end{proof}

\begin{definition}
In the case when $\sigma\in\Phi_{\ell m} (i)  \cup \Phi_{\ell m}(j)$ we define:

$$\tau\star\sigma:=
\begin{cases}
\tau\sigma\tau=\sigma\tau\sigma & \text{ if } \ \ \ 1\notin\{i,j\}\\
\tau\sigma=\sigma\tau & \text { if } \ \ \ 1\in\{i,j\}
\end{cases}$$
One easily shows that $\tau\star\sigma$ is well defined and that $\tau\star\sigma\in\Phi_{\ell m} (i)  \cup \Phi_{\ell m}(j)$.
\end{definition}

\begin{lemma} \label{lem-symm} Let $v$ be generic, $z \in {\mathbb Z}^{\frac{n(n-1)}{2}}$, and $\sigma \in \Phi_{\ell m}$, $\ell \neq m$.
\begin{itemize}
\item[(i)] If $\sigma \notin \Phi_{\ell m} (i)  \cup \Phi_{\ell m} (j)$ then we have:

\begin{itemize}
\item[(a)] $ {\rm ev}(\bar{v}) e_{\ell m} (\sigma( v+\tau(z)))  =  {\rm ev}(\bar{v}) e_{\ell m} (\sigma( v+z))$ and ${\mathcal D}^{\bar{v}} e_{\ell m} (\sigma( v+\tau(z)))  =  - {\mathcal D}^{\bar{v}} e_{\ell m} (\sigma( v+z))$. In particular, $ {\mathcal D}^{\bar{v}} e_{\ell m} (\sigma( v+z)) = 0$ if $\tau (z) = z$.

\item[(b)] ${\rm ev}(\bar{v}) ((x-y)e_{\ell m} (\sigma( v+\tau(z)))  =  {\rm ev}(\bar{v})((x-y) e_{\ell m} (\sigma( v+z))) = 0$ and ${\mathcal D}^{\bar{v}} ((x-y) e_{\ell m} (\sigma( v+\tau(z))) ) =   {\mathcal D}^{\bar{v}} ((x-y)e_{\ell m} ( \sigma( v+z)))$. 
\end{itemize}

\item[(ii)]   If $\sigma \in \Phi_{\ell m} (i)  \cup \Phi_{\ell m} (j)$ then $e_{\ell m} (\tau\star\sigma(v + z))= e_{\ell m} (\sigma\tau(v + z))$. In particular.
\begin{itemize}
\item[(a)] If $\tau (z) \neq z$, $ {\rm ev}(\bar{v}) e_{\ell m} (\tau\star \sigma( v+\tau(z)))  =  {\rm ev}(\bar{v}) e_{\ell m} (\sigma( v+z))$ and ${\mathcal D}^{\bar{v}} e_{\ell m} (\tau\star \sigma( v+\tau(z)))  =  - {\mathcal D}^{\bar{v}} e_{\ell m} (\sigma( v+z))$. 

\item[(b)] ${\rm ev}(\bar{v})((x-y) e_{\ell m} (\tau \star\sigma( v+\tau(z))))  =  - {\rm ev}(\bar{v})((x-y) e_{\ell m} (\sigma( v+z)))$ and ${\mathcal D}^{\bar{v}} ((x-y) e_{\ell m} ( \tau\star \sigma( v+\tau(z))))  =   {\mathcal D}^{\bar{v}} ((x-y)e_{\ell m} ( \sigma( v+z)))$. 
\end{itemize}
\end{itemize}
\end{lemma}
\begin{proof}
The statements follow by direct verification using Definition \ref{definition of coefficients e_rs}. 
\end{proof}

\subsection{Proof of Proposition \ref{compatible}} \label{proof-compatible}
For part (i) we use 
\begin{eqnarray*}
{\mathcal D}^{\bar{v}} ((x-y) E_{rs} T(v+z)) & = & \sum_{\sigma\in\Phi_{rs}}{\mathcal D}^{\bar{v}} ((x-y)e(\sigma(v+z))){\rm ev} (\overline{v}) T(v+z+\sigma(\varepsilon_{rs}))\\
& + & \sum_{\sigma\in\Phi_{rs}}{\rm ev} (\overline{v}) ((x-y)e(\sigma(v+z))){\mathcal D} T(\bar{v}+z+\sigma(\varepsilon_{rs})).
\end{eqnarray*}

The same formula holds for ${\mathcal D}^{\bar{v}} ((x-y) E_{rs} T(v+\tau(z)))$ after replacing $z$ with $\tau(z)$ on the right hand side.
If $\sigma\notin\Phi_{rs}(i)\cup\Phi_{rs}(j)$ then, $\tau(z+\sigma(\varepsilon_{rs}))=\tau(z)+\sigma(\varepsilon_{rs})$ which implies that ${\mathcal D} T(\bar{v}+z+\sigma(\varepsilon_{rs}))=-{\mathcal D} T(\bar{v}+\tau(z)+\sigma(\varepsilon_{rs}))$ and ${\rm ev} (\overline{v}) T(v+z+\sigma(\varepsilon_{rs}))={\rm ev} (\overline{v}) T(v+\tau(z)+\sigma(\varepsilon_{rs}))$. Thanks to Lemma \ref{lem-symm}(i)(b) the corresponding coefficients in the identity of part (i) are the same. In the case $\sigma\in\Phi_{rs}(i)\cup\Phi_{rs}(j)$ we have $\tau(z+\tau\star\sigma(\varepsilon_{rs}))=\tau(z)+\sigma(\varepsilon_{rs})$, so  ${\mathcal D} T(\bar{v}+\tau(z)+\sigma(\varepsilon_{rs}))=-{\mathcal D}T(\bar{v}+z+\tau\star\sigma(\varepsilon_{rs}))$ and  ${\rm ev} (\overline{v}) T(v+\tau(z)+\sigma(\varepsilon_{rs}))={\rm ev} (\overline{v}) T(v+z+\tau\star\sigma(\varepsilon_{rs}))$ and, now by Lemma \ref{lem-symm}(ii)(b) the coefficients are the same.

The proof of 
 part (ii) is similar.

\subsection{Proof of Proposition \ref{d-t-v-rep}} \label{proof-d-t}

First, we assume that  $g_1 = E_{rs}$ and $g_2 = E_{\ell m}$. The case $r = s$  or $\ell = m$ follows by straightforward computations. For example, if $r=s$, due to the hypothesis $\tau(z)\neq z$, the functions $e_{rr}(v+z)$, $e_{rr}(v+z+\sigma(\varepsilon_{\ell m}))$ and $e_{\ell m}(\sigma(v+z))$ are in $\mathcal{F}_{ij}$. Hence we can apply $\mathcal{D}^{\overline{v}} $ and  $\mbox{ev}(\overline{v})$ to these functions in order to show that the corresponding coefficients in (\ref{representation for derivative tableaux}) coincide. Note that we need to use Lemma \ref{dif operaator in functions}(iii).

Assume now  that $r \neq s$ and $\ell \neq m$. 
Since $\tau (z) \neq z$, we have  $\overline{\Phi}_{rs} \subset  {\Phi}_{rs} (i) \cup  {\Phi}_{rs} (j)$ and $\overline{\Phi}_{\ell m} \subset  {\Phi}_{\ell m} (i) \cup  {\Phi}_{\ell m} (j)$. For convenience we will use the following convention   for $\sigma_1 \in \Phi_{rs}$ and $\sigma_2 \in \Phi_{\ell m}$:
\begin{eqnarray*}
e_{rs } (\sigma_1) = e_{rs} (\sigma_1 (v+z)), &&e_{rs} (\sigma_1, \sigma_2) = e_{rs} (\sigma_1 (v + z + \sigma_2 (\varepsilon_{\ell m}))), \\
e_{\ell m} (\sigma_2) = e_{\ell m} (\sigma_2 (v+z)), &&e_{\ell m} (\sigma_1, \sigma_2) = e_{\ell m} (\sigma_2 (v + z + \sigma_1 (\varepsilon_{r s}))) .
\end{eqnarray*}
We also set $\bar{e}_{\ell m} = (x-y)e_{\ell m}, \bar{e}_{rs } = (x-y)e_{r s}$ and similarly introduce $\bar{e}_{rs} (\sigma_1, \sigma_2) $ and $\bar{e}_{\ell m} (\sigma_1, \sigma_2)$.  Furthermore, we set 
\begin{eqnarray*}
T(\sigma_1 + \sigma_2) = T(v+z+\sigma_1 (\varepsilon_{r s}) + \sigma_2 (\varepsilon_{\ell m})),
\end{eqnarray*}

and
\begin{eqnarray*}
  \mathcal{D} T(\sigma_1 + \sigma_2) = \mathcal{D} T(\bar{v}+z+\sigma_1 (\varepsilon_{r s}) + \sigma_2 (\varepsilon_{\ell m})).
\end{eqnarray*}

Finally, for $\sigma_1 \in \Phi_{rs}$ and $\sigma_2 \in \Phi_{\ell m}$ we set:

\begin{eqnarray*}
L_1(\sigma_1, \sigma_2) & = & 
\begin{cases}
\mathcal{D}^{\overline{v}} ( e_{\ell m} (\sigma_2)   e_{rs} (\sigma_1, \sigma_2 ) T(\sigma_1 + \sigma_2)), \mbox{ if } \sigma_2 \notin \overline{\Phi}_{\ell m}  \\
\mathcal{D}^{\overline{v}} (e_{\ell m} (\sigma_2) )  \mathcal{D}^{\overline{v}}( \bar{e}_{rs} (\sigma_1, \sigma_2 ) T(\sigma_1 + \sigma_2)), \mbox{ if } \sigma_2 \in \overline{\Phi}_{\ell m},
\end{cases}\\
L_2(\sigma_1, \sigma_2) & = & \begin{cases}
\mathcal{D}^{\overline{v}} ( e_{rs} (\sigma_1)   e_{\ell m} (\sigma_1, \sigma_2 ) T(\sigma_1 + \sigma_2)), \mbox{ if } \sigma_1 \notin \overline{\Phi}_{rs }  \\
\mathcal{D}^{\overline{v}} (e_{rs } (\sigma_1) )  \mathcal{D}^{\overline{v}}( \bar{e}_{\ell m} (\sigma_1, \sigma_2 ) T(\sigma_1 + \sigma_2)), \mbox{ if } \sigma_1 \in \overline{\Phi}_{rs},
\end{cases}\\
R_1(\sigma_1, \sigma_2) & = & e_{\ell m} (\sigma_2)   e_{rs} (\sigma_1, \sigma_2 ) T(\sigma_1 + \sigma_2),\\
R_2(\sigma_1, \sigma_2) & = &  e_{rs} (\sigma_1)   e_{\ell m} (\sigma_1, \sigma_2 ) T(\sigma_1 + \sigma_2).
\end{eqnarray*}

Applying Lemma \ref{dij-commute} we obtain
\begin{equation} \label{eq-left}
E_{rs}( E_{\ell m} (\mathcal{D}T(\bar{v} + z))) - E_{\ell m}(E_{rs} (\mathcal{D}T(\bar{v} + z))) = \sum_{\sigma_1,\sigma_2} (L_1(\sigma_1, \sigma_2) - L_2 (\sigma_1, \sigma_2))
\end{equation}
and 
\begin{equation} \label{eq-right}
[E_{rs}, E_{\ell m}] (\mathcal{D}T(\bar{v} + z)) =  \mathcal{D}^{\overline{v}}( [E_{rs}, E_{\ell m}] T(v + z)) =
\end{equation}
$$
=   \mathcal{D}^{\overline{v}}  \left( \sum_{\sigma_1,\sigma_2} (R_1(\sigma_1, \sigma_2) - R_2 (\sigma_1, \sigma_2))\right).
$$
The sums in  (\ref{eq-left}) and  (\ref{eq-right}) run over $(\sigma_1, \sigma_2) \in \Phi_{rs} \times \Phi_{\ell m}$. Our goal is to show that the right hand sides of (\ref{eq-left}) and (\ref{eq-right}) coincide. By the definitions of $L_i$ and $R_i$ we have that $L_1(\sigma_1,\sigma_2) = \mathcal{D}^{\bar{v}} (R_1(\sigma_1,\sigma_2))$ if $\sigma_2 \notin \overline{\Phi}_{\ell m}$ and $L_2(\sigma_1,\sigma_2) =\mathcal{D}^{\bar{v}} (R_2(\sigma_1,\sigma_2))$ if $\sigma_1 \notin \overline{\Phi}_{rs}$. Next we show that $L_1(\sigma_1,\sigma'_2) = \mathcal{D}^{\bar{v}}(R_1(\sigma_1,\sigma'_2))$ if  $\sigma_1 \notin {\Phi}_{rs} (i) \cup {\Phi}_{rs} (j) $ and  $\sigma_2' \in \overline{\Phi}_{\ell m}$. In this case we have   $\tau (\bar{v} + z + \sigma_1 (\varepsilon_{rs}) + \sigma_2' (\varepsilon_{\ell m}))  = \bar{v} + z + \sigma_1 (\varepsilon_{rs}) + \sigma_2' (\varepsilon_{\ell m})$  and, hence, $ \mathcal{D} T(\sigma_1 + \sigma_2') = 0$. Therefore,
\begin{eqnarray*}
L_1(\sigma_1,\sigma_2') & = &  \mathcal{D}^{\overline{v}} (e_{\ell m} (\sigma_2') )  \mathcal{D}^{\overline{v}}( \bar{e}_{rs} (\sigma_1, \sigma_2' ) T(\sigma_1 + \sigma_2'))\\
&  = &  \mathcal{D}^{\overline{v}} (e_{\ell m} (\sigma_2') )  \mathcal{D}^{\overline{v}}( \bar{e}_{rs} (\sigma_1, \sigma_2' )) {\rm ev} (\bar{v})T(\sigma_1 + \sigma_2')\\
&  = &  \mathcal{D}^{\overline{v}} (e_{\ell m} (\sigma_2') )  {\rm ev} (\overline{v})( e_{rs} (\sigma_1, \sigma_2' )) {\rm ev} (\bar{v})T(\sigma_1 + \sigma_2')\\
&  = &  \mathcal{D}^{\overline{v}} (e_{\ell m} (\sigma_2')  e_{rs} (\sigma_1, \sigma_2' )) {\rm ev} (\bar{v})T(\sigma_1 + \sigma_2')\\
&  = &  \mathcal{D}^{\bar{v}}(R_1 (\sigma_1, \sigma_2')).
\end{eqnarray*}
The first equality follows from the definition of $L_1$, while the second follows from $ \mathcal{D} T(\sigma_1 + \sigma_2') = 0$.  The third equality follows from the fact that $\bar{e}_{rs} (\sigma_1, \sigma_2' ) $ is in ${\mathcal F}_{ij}$ by Lemma \ref{poles of order at most 1}. The forth one is a consequence of $ \mathcal{D}^{\overline{v}} e_{rs} (\sigma_1, \sigma_2' ) =0$ (by Lemma \ref{lem-symm}(i)(a)), and finally for the last equality we use $ \mathcal{D}T(\sigma_1 + \sigma_2') = 0$. Similarly one can show that $L_1(\sigma_1',\sigma_2) = \mathcal{D}^{\bar{v}}(R_1(\sigma_1',\sigma_2))$ if  $\sigma'_1 \in \overline{\Phi}_{rs} $ and $\sigma_2  \notin {\Phi}_{\ell m} (i) \cup {\Phi}_{\ell m} (j)$. We next note that if  $\sigma_2' \in \overline{\Phi}_{\ell m}$ and $\sigma_1 \in {\Phi}_{rs} (i) \cup {\Phi}_{rs} (j) $ then either $\sigma_1 \in \overline{\Phi}_{rs}$ or $\tau\star \sigma_1 \in \overline{\Phi}_{rs}$; and if  $\sigma_1' \in \overline{\Phi}_{rs}$ and $\sigma_2 \in {\Phi}_{\ell m} (i) \cup {\Phi}_{\ell m} (j) $ then either $\sigma_2 \in \overline{\Phi}_{\ell m}$ or $\tau\star \sigma_2 \in \overline{\Phi}_{\ell m}$. To complete the proof it is sufficient to prove that for $\sigma_1' \in \overline{\Phi}_{rs}$ and $\sigma_2' \in \overline{\Phi}_{\ell m}$ we have $L =  \mathcal{D}^{\overline{v}} (R)$, where 
$$L  =
 \sum_{(\sigma_1, \sigma_2) \in \Phi_{(\sigma_1',\sigma_2')}}\left(L_1(\sigma_1, \sigma_2)+ L_1(\tau\star \sigma_1, \sigma_2)- L_2(\sigma_1, \sigma_2) - L_2(\sigma_1, \tau \star\sigma_2)\right),$$
$$R = \sum_{(\sigma_1, \sigma_2) \in \Phi_{(\sigma_1',\sigma_2')}}\left(R_1(\sigma_1, \sigma_2)+ R_1(\tau\star \sigma_1, \sigma_2)- R_2(\sigma_1, \sigma_2) - R_2(\sigma_1, \tau \star\sigma_2)\right).$$

We then note that $\sum_{\Phi_{(\sigma_1',\sigma_2')}}\left(R_1(\sigma_1, \sigma_2) - R_2(\sigma_1, \sigma_2)\right)=0$ thanks to Lemma \ref{lm-rs-ident}. 

We can simplify the expansion of $L$ using various identities. We use first that $z + \tau \star\sigma_1( \varepsilon_{rs}) + \sigma_2( \varepsilon_{\ell m})= z + \sigma_1( \varepsilon_{rs}) + \tau \star\sigma_2( \varepsilon_{\ell m})$, which implies $T(\tau\star \sigma_1+\sigma_2) = T(\sigma_1+\tau \star\sigma_2)$ and, in particular, $ {\rm ev}(\bar{v})T(\tau \star\sigma_1+\sigma_2) =  {\rm ev}(\bar{v})T(\sigma_1+\tau \star\sigma_2)$ and $  \mathcal{D} T(\tau \star\sigma_1+\sigma_2) =   \mathcal{D} T(\sigma_1+\tau\star \sigma_2)$. We also use  that  $ \mathcal{D}T(\sigma_1+\sigma_2)  = - \mathcal{D} T(\tau\star \sigma_1+\sigma_2)$ and $ {\rm ev}(\bar{v})T(\sigma_1+\sigma_2)  =  {\rm ev}(\bar{v})T(\tau \star\sigma_1+\sigma_2)$. On the other hand, by Lemma \ref{lem-symm}(ii)(b) we have that $ \mathcal{D}^{\overline{v}} (\bar{e}_{rs} (\sigma_1, \sigma_2)) = \mathcal{D}^{\overline{v}} (\bar{e}_{rs} (\tau\star \sigma_1, \sigma_2)) $, $ \mathcal{D}^{\overline{v}} (\bar{e}_{\ell m} (\sigma_1, \sigma_2)) = \mathcal{D}^{\overline{v}} (\bar{e}_{\ell m } (\sigma_1, \tau\star \sigma_2)) $, $  {\rm ev}(\bar{v}) (\bar{e}_{rs} (\sigma_1, \sigma_2)) =  - {\rm ev}(\bar{v}) (\bar{e}_{rs} (\tau\star \sigma_1, \sigma_2))$, and $  {\rm ev}(\bar{v}) (\bar{e}_{\ell m} (\sigma_1, \sigma_2)) =  - {\rm ev}(\bar{v}) (\bar{e}_{\ell m } (\sigma_1, \tau\star \sigma_2)) $. All these identities reduce  $L$ to twice the sum of 
\begin{eqnarray*} 
 \mathcal{D}^{\overline{v}} (e_{\ell m} (\sigma_2))  \mathcal{D}^{\overline{v}} (\bar{e}_{rs} (\tau\star \sigma_1, \sigma_2)) - \mathcal{D}^{\overline{v}} (e_{rs} (\sigma_1)) \mathcal{D}^{\overline{v}} (\bar{e}_{\ell m} (\sigma_1, \tau\star \sigma_2))  {\rm ev}(\bar{v})T(\sigma_1+\sigma_2) \\
\mathcal{D}^{\overline{v}} (e_{\ell m} (\sigma_2))   {\rm ev}(\bar{v})(\bar{e}_{rs} (\tau \star\sigma_1, \sigma_2)) - \mathcal{D}^{\overline{v}} (e_{rs} (\sigma_1))  {\rm ev}(\bar{v}) (\bar{e}_{\ell m} (\sigma_1, \tau\star \sigma_2))  \mathcal{D} T(\sigma_1+\sigma_2).
\end{eqnarray*}
over all $(\sigma_1, \sigma_2) \in \Phi_{(\sigma_1',\sigma_2')}$.

Let $\{(\sigma_1^{(p)},\sigma_{2}^{(p)})\}_{p=1}^{t}$ be the set of all distinct pairs of permutations in $\Phi_{(\sigma_1,\sigma_{2})}$.
For each $p=1,\ldots,t$  define the following functions:
\begin{eqnarray*}
 f_{2p} =  e_{\ell m}(\sigma_{2}^{(p)}), && \ \ 
f_{2p-1} =e_{rs}(\sigma_{1}^{(p)}),\\
g_{2p-1} = e_{rs}(\tau\sigma_{1}^{(p)},\sigma_{2}^{(p)}),& &\ \ 
g_{2p} =-e_{\ell m}(\sigma_{1}^{(p)},\tau\sigma_{2}^{(p)}).
\end{eqnarray*}

Note that $e_{rs} (\tau\sigma_1^{(p)}, \sigma_2^{(p)}) = e_{rs} (\sigma_1^{(p)}, \sigma_2^{(p)})^{\tau}$ and $e_{\ell m} (\sigma_1^{(p)}, \tau \sigma_2^{(p)}) = e_{\ell m} (\sigma_1^{(p)}, \sigma_2^{(p)})^{\tau}$ thanks to Lemma \ref{lem-symm}(ii). We finally apply Lemma \ref{dv-formulas} to the set of functions $f_{p},g_{p}$, $p=1,\ldots, 2t$, and obtain $L =  \mathcal{D}^{\overline{v}} (R)$. Note that the  hypothesis  $\sum_{p=1}^{2t}f_{p}g_{p}^{\tau}=0$  of Lemma \ref{dv-formulas}   holds by Lemma \ref{lm-rs-ident}.


\begin{thebibliography}{20}


\bibitem{Deo} V. Deodhar, On a construction of representations and a problem of Enright, {Invent. Math.} {57} (1980), 101--118.



\bibitem{DFO2} Y. Drozd, S. Ovsienko, V. Futorny, Irreducible weighted $\mathfrak{sl}(3)$-modules, {Funksionalnyi Analiz i Ego Prilozheniya},  23 (1989), 57--58.

\bibitem{DFO1} Y. Drozd, V. Futorny, S. Ovsienko,   Gelfand-Tsetlin modules over Lie algebra $\mathfrak{sl}(3)$, {Contemp. Math.} {131} (1992) 23--29.

\bibitem{DFO3} Y. Drozd, S. Ovsienko, V. Futorny,  Harish-Chandra subalgebras and Gelfand-Zetlin modules, {Math. and Phys. Sci.} {424} (1994), 72--89.

\bibitem{FJMM} B. Feigin, M. Jimbo, T. Miwa, E. Mukhin,
    Quantum toroidal $\mathfrak{gl}_1$-algebra: plane partitions, {Kyoto J.Math.} {52} (2012), 621--659.
    

\bibitem{Fe} S. Fernando, Lie algebra modules with finite dimensional weight spaces I, {Trans. Amer. Math. Soc.} {322} (1990), 757--781.

\bibitem{FM} T. Fomenko, A. Mischenko, Euler equation on finite-dimensional Lie
groups, {Izv. Akad. Nauk SSSR, Ser. Mat.} {42} (1978), 396--415.

\bibitem{FO1} V. Futorny, S. Ovsienko,
Galois orders in skew monoid rings, {J. Algebra}, {324} (2010), 598--630

\bibitem{FO2} V. Futorny, S. Ovsienko, Fibers of characters in Gelfand-Tsetlin categories,  {Trans. Amer. Math. Soc.} {366} (2014), 4173--4208.


\bibitem{FGR} V. Futorny, D. Grantcharov, L. E. Ramirez, Classification of irreducible Gelfand-Tsetlin modules for $\mathfrak{sl}(3)$. In progress.



\bibitem{FGR2} V. Futorny, D. Grantcharov, L. E. Ramirez, Generic irreducible Gelfand-Tsetlin modules for $\mathfrak{sl}(n)$. In progress.


\bibitem{GG} I. Gelfand, M. Graev, Finite-dimensional irreducible representations of the unitary and complete linear group and special functions associated with them. (Russian) {Izv. Akad. Nauk SSSR Ser. Mat.} {29} (1965), 1329--1356.




\bibitem{GT} I. Gelfand, M. Tsetlin, Finite-dimensional representations of the group of unimodular matrices, {Doklady Akad. Nauk SSSR (N.s.)}, {71} (1950), 825--828.

\bibitem{Gr1} M. Graev, Infinite-dimensional representations of the Lie algebra $\gl(n, \mathbb{C})$
related to complex analogs of the Gelfand-Tsetlin patterns and
general hypergeometric functions on the Lie group $GL(n,
\mathbb{C})$, {Acta Appl. Mathematicae}, {81} (2004), 93--120.

\bibitem{Gr2} M. Graev, A continuous analogue of Gelfand-Tsetlin schemes and a realization of the
principal series of irreducible unitary representations of the group $GL(n,\mathbb{C})$ in the space
of functions on the manifold of these schemes, {Dokl. Akad. Nauk.} {412 no.2} (2007), 154--158.

\bibitem{KW-1}  B. Kostant,  N. Wallach, Gelfand-Zeitlin theory from the perspective of
classical mechanics I, In Studies in Lie Theory Dedicated
to A. Joseph on his Sixtieth Birthday, {Progress in
Mathematics},  {243} (2006),   319--364.


\bibitem{KW-2}  B. Kostant,  N. Wallach, Gelfand-Zeitlin theory from the perspective of
classical mechanics II. In The Unity of Mathematics In
Honor of the Ninetieth Birthday of I. M. Gelfand, {Progress
in Mathematics},  {244} (2006), 387--420.


\bibitem{LP1}  F. Lemire, J. Patera, Formal analytic continuation of Gelfand finite-dimensional representations  of $gl(n,\mathbb{C})$, {J. Math. Phys.} {20} (1979), 820--829.


\bibitem{LP2} F. Lemire, J. Patera, Gelfand representations of $sl(n, \mathbb{C})$,{ Algebras Groups Geom.} {2} (1985), 14--166.



\bibitem{M} O. Mathieu, Classification of irreducible weight modules, {Ann. Inst. Fourier}, {50} (2000), 537--592.



\bibitem{Maz1} V. Mazorchuk, Tableaux realization of generalized Verma modules, {Can. J. Math.}  {50} (1998), 816--828.


\bibitem{Maz2} V. Mazorchuk, On categories of Gelfand-Zetlin modules, {Noncommutative Structures in Mathematics and Physics}, Kluwer Acad. Publ, Dordrecht (2001),  299--307.

\bibitem{m:gtsb} A. Molev,  Gelfand-Tsetlin bases for classical Lie algebras, {Handbook of Algebra, Vol. 4, (M. Hazewinkel, Ed.), Elsevier,} (2006), 109--170.



\bibitem{Ovs} S. Ovsienko, Finiteness
statements for Gelfand-Zetlin modules, {Third International Algebraic Conference in the Ukraine (Ukrainian)},  Natsional. Akad. Nauk Ukrainy, Inst. Mat., Kiev, (2002), 323--338.

\bibitem{Ramirez} L. E. Ramirez, Combinatorics of irreducible Gelfand-Tsetlin  $\mathfrak{sl}(3)$-modules, Algebra Discrete Math. {14}  (2012) 276--296.

\bibitem{Vi}  E. Vinberg, On certain commutative
subalgebras of a universal enveloping
algebra, {Math. USSR Izvestiya}, {36} (1991), 1--22.

\bibitem{Zh} D. Zhelobenko, Compact Lie groups and their representations, {Transl. Math. Monographs, AMS}, {40} (1974)

\end{thebibliography}
\end{document}